\documentclass[11pt]{article}

\usepackage[english]{babel}
\usepackage[T1]{fontenc}

\usepackage{amsfonts}
\usepackage{amsmath}
\usepackage{amssymb}
\usepackage{amsthm}
\usepackage{mathtools}
\usepackage{stmaryrd}
\usepackage{dsfont}
\usepackage{graphicx}
\usepackage{hyperref}
\usepackage{titlesec}
\usepackage{mathabx}
\usepackage{enumitem,tocloft}
\usepackage[top=2.5cm, bottom=2.5cm, left=3cm, right=3cm]{geometry}
\usepackage{xcolor}
\usepackage{fourier-orns}
\usepackage[utf8]{inputenc}
\usepackage{bm}
\usepackage{bbm}
\usepackage{framed}
\usepackage{supertabular}
\hypersetup{linktocpage,breaklinks=true}
\usepackage{mathrsfs}
\usepackage{csquotes}
\usepackage{bbold}

\SetSymbolFont{stmry}{bold}{U}{stmry}{m}{n}

\usepackage[backend=bibtex, style=alphabetic, maxbibnames=6, sorting=nyt]{biblatex}

\theoremstyle{plain}
\newtheorem{theorem}{Theorem}[section]
\newtheorem{corollary}[theorem]{Corollary}

\newtheorem{theoremletter}{Theorem}

\newtheorem{corollaryletter}[theoremletter]{Corollary}
\newtheorem{propositionletter}[theoremletter]{Proposition}

\newtheorem{lemma}[theorem]{Lemma}
\newtheorem{proposition}[theorem]{Proposition}

\newtheorem{claim}[theorem]{Claim}
\theoremstyle{definition}
\newtheorem{definition}[theorem]{Definition}
\newtheorem{example}[theorem]{Example}
\newtheorem{question}[theorem]{Question}
\newtheorem{remark}[theorem]{Remark}
\newtheorem{fact}[theorem]{Fact}

\newenvironment{cproof}{\begin{proof}[Proof of the
		claim]}{\end{proof}}

\numberwithin{equation}{section}

\newcommand{\field}{\mathbb{f}}

\newcommand{\R}{\mathbb{R}}

\newcommand{\Z}{\mathbb{Z}}	

\newcommand{\N}{\mathbb{N}}	

\newcommand{\fsym}[1]{\mathsf{FSym}(#1)}

\newcommand{\sym}[1]{\mathsf{Sym}(#1)}

\newcommand{\shuf}[1]{\mathsf{Shuffler}(#1)}

\newcommand{\shufn}[2]{\mathsf{Shuffler}^{\circ {#1}}(#2)}

\newcommand{\juggler}[2]{\mathsf{Shuffler}_{#1}(#2)}

\newcommand{\jugglern}[3]{\mathsf{Shuffler}_{#2}^{\circ {#1}}(#3)}

\newcommand{\cloner}[1]{\mathsf{Cloner}_{\field}(#1)}

\newcommand{\clonern}[2]{\mathsf{Cloner}_{\field}^{\circ {#1}}(#2)}

\newcommand{\designer}[1]{\mathsf{Designer}_{F}(#1)}

\newcommand{\designern}[2]{\mathsf{Designer}_{F}^{\circ {#1}}(#2)}

\newcommand{\upcloner}[1]{\mathsf{Upcloner}_{\field}(#1)}

\newcommand{\fol}[1]{\textup{F\o l}_{#1}} 

\newcommand{\supp}{\mathrm{supp}\ } 

\newcommand{\prof}[1]{j_{1,#1}}
\newcommand{\profp}[1]{j_{p,#1}}
\newcommand{\folp}[1]{\textup{F\o l}_{p,#1}}

\newcommand{\halo}{\mathscr{L}}

\makeatletter
\newcommand*{\defeq}{\mathrel{\rlap{%
                     \raisebox{0.3ex}{$\m@th\cdot$}}%
                     \raisebox{-0.3ex}{$\m@th\cdot$}}%
                     =}
\makeatother

\begin{document}

\begin{titlepage}
\setcounter{page}{1}
\title{Isoperimetric profiles of lamplighter-like groups}
\author{\small{Corentin Correia and Vincent Dumoncel}}

\date{\today}
\maketitle

\begin{abstract}
Given a finitely generated amenable group $H$ satisfying some mild assumptions, we relate isoperimetric profiles of the lampshuffler group $\shuf{H}=\fsym{H}\rtimes H$ to those of $H$. Our results are sharp for all exponential growth groups for which isoperimetric profiles are known, including Brieussel-Zheng groups. This refines previous estimates obtained by Erschler and Zheng and by Saloff-Coste and Zheng. \par
The most difficult part is to find an optimal upper bound, and our strategy consists in finding suitable lamplighter subgraphs in lampshufflers. This novelty applies more generally for many examples of \textit{halo products}, a class of groups introduced recently by Genevois and Tessera as a natural generalisation of wreath products.\par Lastly, we also give applications of our estimates on isoperimetric profiles to the existence problem of regular maps between such groups.
\end{abstract}

\smallskip

	
	\smallskip
	
	{
		\small	
		\noindent\textbf{{MSC-classification:}}	
		Primary 20F65; Secondary 20F69.
	}


\tableofcontents

\section{Introduction}

It is a recurrent theme in geometric group theory to understand the collection of all maps between two given finitely generated groups that are compatible with their large-scale geometries. Such collections include for instance quasi-isometries, coarse embeddings, and more generally regular maps. The motivation behind this program is that the large-scale geometry of a group is in fact deeply related to its algebraic structure. 

\smallskip

\sloppy Several milestones have been achieved in the study of the quasi-isometries of many classes of groups, among which abelian free groups, lamplighters over virtually cyclic groups and other $\text{SOL}$-like groups~\cite{EFW12, EFW13}, Baumslag-Solitar groups~\cite{FM98, FM99, Why01}, or lamplighters over one-ended groups~\cite{GT24b}.

\smallskip

If one rather wants to exclude the existence of such maps between some spaces, an efficient strategy is to use invariants, or even monotonuous quantities, that are computable in practice. Numerous invariants have been introduced for quasi-isometries, including 
 the isoperimetric profiles which are monotonuous under regular maps, see~\cite[Theorem~5.5]{DKLMT22}. We refer the reader to Section~\ref{sec:appQIandRegularMaps} for other examples of invariants (asymptotic dimension, volume growth).

\smallskip

The aim of the paper is to compute isoperimetric profiles of lampshufflers, and more generally of the so-called \textit{halo products} introduced in~\cite{GT24a}. 

\paragraph{Isoperimetric profiles and F\o lner functions.} For a finitely generated group $G$ with a finite generating set $S_{G}$, and given $p\geq 1$, its $\ell^p$-\textit{isoperimetric profile} is the function $\profp{G}\colon \N\longrightarrow\R_{+}$ given by 
\begin{equation*}
    \profp{G}(n) \defeq\sup_{\substack{f\colon G\to\R_+\\|\supp{f}|\leq n}}{\frac{\|f\|_p}{\|\nabla f\|_p}}
\end{equation*}
where the support of $f\colon G\longrightarrow\R_+$ is $\supp{f}\defeq\lbrace g\in G : f(g)\neq 0\rbrace$ and the $\ell^p$-norm of its gradient is defined by 
\begin{equation*}
   \|\nabla f\|_{p}^{p}\defeq\sum_{g\in G,\; s\in S_{G}}\left|f(g)-f(gs)\right|^p. 
\end{equation*}

\begin{remark}\label{rem:ConventionProfile}
    We warn the reader that many authors introduce the $\ell^p$-isoperimetric profile with $\|\cdot\|_p^p$ instead of $\|\cdot\|_p$ in the definition, so their $\ell^p$-isoperimetric profile is $\profp{G}(x)^p$ with our conventions.
\end{remark}

The $\ell^p$-isoperimetric profile of a group $G$ is the generalised inverse of its $\ell^p$-\textit{F\o lner function} $\folp{G}\colon\N\longrightarrow \R_{+}$, defined as
 \begin{equation*}
     \folp{G}(n) \defeq \inf\left\lbrace |\supp{f}| : \frac{\|\nabla f\|_p}{\|f\|_p} \le \frac{1}{n} \right\rbrace.
 \end{equation*}

In the case $p=1$, these functions are simply called \textit{isoperimetric profile} and \textit{F\o lner function}, and have a simpler definition (up to asymptotic behaviour), namely
\begin{equation*}
    \prof{G}(n) \defeq \sup_{|A|\le n}\frac{|A|}{|\partial_{G} A|}\;\text{ and }\;\fol{G}(n) \defeq \inf\left\lbrace |A| : \frac{|\partial_{G}A|}{|A|} \le \frac{1}{n} \right\rbrace,
\end{equation*}
where $\partial_{G}A \defeq AS_{G}\setminus A=\lbrace g\in G\setminus A : \exists s\in S_{G}, \exists h\in A, g=hs\rbrace$ is the boundary of $A$ in $G$. Note that we only find the $\ell^1$-F\o lner function in the literature. In this paper we introduce the more general $\ell^p$ versions for $p\ge 1$.

Without loss of generality, we may and do assume that the $\ell^p$-isoperimetric profile and the $\ell^p$-F\o lner function are real inverses of each other, and not only generalised inverses; see Remark~\ref{rem:InverseEachOther}.

\smallskip

Notice that the $\ell^p$-isoperimetric profile of a finitely generated group is bounded if and only if the group is not amenable. Therefore, we will only be interested in $\ell^p$-isoperimetric profiles of amenable groups. The asymptotic behaviour of the $\ell^p$-isoperimetric profile is, somehow, a measurement of its amenability; the faster it goes to infinity, the "more amenable" the group is.

\smallskip

Among amenable groups, the $\ell^p$-isoperimetric profile (or equivalently the $\ell^p$-F\o lner function) has been computed for many finitely generated groups. Given $p\ge 1$, we have for instance:
\begin{itemize}
    \item $\profp{G}(n) \simeq n^{\frac{1}{d}}$ if $G$ has polynomial growth of degree $d\ge 1$;
    \item $\profp{G}(n)\simeq \ln(n)$ for $G=\text{BS}(1,k)$ for $k\ge 2$, or $G=F\wr\Z$, where $F$ is a non-trivial finite group;
    \item $\profp{G}(n) \simeq \ln(n)$ for any polycyclic group $G$ with exponential growth~\cite{Pit95, Pit00}, or more generally any exponential growth group within the class GES of Tessera~\cite[Corollary 5]{Tes13};
    \item $\profp{F\wr N}(n)\simeq (\ln(n))^{\frac{1}{d}}$ with $F$ finite, and $N$ having polynomial growth of degree $d\ge 1$~\cite{Ers03};
    \item for any non-decreasing function $f\colon \R_{+}\longrightarrow \R_{+}$ such that $x\longmapsto \frac{x}{f(x)}$ is non-decreasing, Brieussel and Zheng constructed in~\cite[Theorem~1.1]{BZ21} a finitely generated group $H$ with exponential volume growth having isoperimetric profile $\profp{H}(x)\simeq \frac{\ln(x)}{f(\ln(x))}$; we will refer to such a group as a Brieussel-Zheng's group. 
\end{itemize}

In fact, isoperimetric profiles and F\o lner functions have been studied in the more general framework of bounded degree graphs; see Section~\ref{sec:preliminaries} for details. For now, let us simply mention that, in this setup, Erschler's estimates for $\ell^1$-F\o lner functions of lamplighter graphs~\cite{Ers03,Ers06}  will be a key ingredient in our strategy (see Section~\ref{sec:secondstrategy}).

\smallskip

It is a well-known fact that if $p>q$, then $\profp{G}(x)\preccurlyeq j_{q,G}(x)$, see~Lemma~\ref{lem:MonotonuousProfile}. Moreover, a stronger phenomenon is conjectured: $\profp{G}(x)\simeq j_{q,G}(x)$ for all $p,q\ge 1$.

\paragraph{Lampshuffler groups.} As a starting point of our work, let us first focus on \textit{lampshuffler groups}. Later in the introduction, we will present the more general notion of halo products, constructed in a similar way to wreath products (as lampshufflers). 

\smallskip

Given a group $H$, the \textit{lampshuffler group} over $H$ is the semi-direct product 
\begin{equation*}
    \shuf{H} \defeq \fsym{H}\rtimes H
\end{equation*}
where $\fsym{H}$ is the group of finitely supported bijections $H\longrightarrow H$, and where $H$ acts on the latter as $(h\cdot\sigma)(x)\defeq h\sigma(h^{-1}x)$. These groups already appeared several times in the literature, in relations with many topics of interest in group theory, see for instance~\cite{Yad09, HO16, BZ19, EZ21, SCZ21, GT24a, Sil24}.

\smallskip

Let us first present what is known about profiles of lampshufflers and our results for this class of groups.

\paragraph{Isoperimetric profiles of lampshufflers.} In~\cite[Theorem~6.8]{SCZ21}, Saloff-Coste and Zheng establish a general lower bound on $\profp{\shuf{H}}$ for a finitely generated group $H$, of the form 
\begin{equation*}
    \profp{H}\left(\frac{\ln(x)}{\ln(\ln(x))}\right) \preccurlyeq \profp{\shuf{H}}(x).
\end{equation*}
Their proof will be useful since we will present a natural generalisation for halo products, see the lower bound in Corollary~\ref{cor:ProfHalo intro}.

\smallskip

From~\cite[Theorem~6.7]{SCZ21}, we also know upper bounds on $\profp{\shuf{H}}$ for general groups $H$, and~\cite[Corollary~1.4]{EZ21} provides a lower bound on $\textup{F\o l}_{1,\shuf{H}}$, equivalently an upper bound on $\prof{\shuf{H}}$.

\begin{theorem}[{\cite[Theorem~6.7]{SCZ21}}]
Let $H$ be a finitely generated group, with a finite generating set $S_{H}$. For $p=1,\; 2$, we have
\begin{equation*}
    \profp{\shuf{H}}(x)\preccurlyeq V_{H,S_{H}}^{-1}\left(\frac{\ln(x)}{\ln(\ln(x))}\right).
\end{equation*}
\end{theorem}
Here, $V_{H,S_{H}}$ refers to the growth function of $H$ with respect to the finite generating set $S_{H}$. Note that the statement in their paper appears with an exponent $p$, since we do not use the same convention; see Remark~\ref{rem:ConventionProfile}.

\begin{theorem}[{\cite[Corollary~1.4]{EZ21}}]
Let $H$ be a finitely generated group, with a finite generating set $S_H$. Then, we have
\begin{equation*}
    \textup{F\o l}_{1,\shuf{H}}(x)\succcurlyeq V_{H,S_H}(x)^{V_{H,S_H}(x)}.
\end{equation*}
\end{theorem}

From these theorems, one can deduce for instance the isoperimetric profile of lampshufflers over polynomial growth groups:
\begin{equation*}
    \prof{\shuf{H}}(x)\simeq\left(\frac{\ln(x)}{\ln(\ln(x))}\right)^{\frac{1}{d}},
\end{equation*}
when $H$ has polynomial growth of degree $d\ge 1$.

\smallskip

For many groups, the lower and upper bounds provided by these results are not the same, so they do not provide precise estimates on the isoperimetric profiles of lampshufflers. Theorem~\ref{th:boundsForProfile intro} below provides finer estimates. It is in fact an application of Corollary~\ref{cor:ProfHalo intro}, that we deduce from Theorem~\ref{th:recallThm intro}, stated in the more general context of halo products. Before defining these groups, let us focus on lampshufflers and the consequences of Theorem~\ref{th:boundsForProfile intro}.

\smallskip

In our statements, saying that $\profp{H}$ satisfies Assumption~$(\star)$ means that $\profp{H}(Cx)=O(\profp{H}(x))$ for any $C>0$.

\begin{theoremletter}[see Corollary~\ref{cor:boundonProf}]\label{th:boundsForProfile intro}
Let $p\geq 1$. Let $H$ be a finitely generated amenable group whose $\ell^p$-isoperimetric profile $\profp{H}$ satisfies Assumption~$(\star)$. Then,
the $\ell^p$-isoperimetric profile $\profp{\shuf{H}}$ of $\shuf{H}$ satisfies
\begin{equation*}
    \profp{H}\left(\frac{\ln(x)}{\ln(\ln(x))}\right) \preccurlyeq \profp{\shuf{H}}(x) \preccurlyeq \prof{H}(\ln(x)). 
\end{equation*}
\end{theoremletter}

Assumption $(\star)$ already appeared several times in the literature, see e.g.~\cite{Ers03, Corr24} for the case $p=1$, and does not seem to be restrictive. In fact, to our knowledge, there is no known example of a finitely generated amenable group whose isoperimetric profiles do not satisfy Assumption~$(\star)$. For instance, it is easy to check that Brieussel-Zheng's groups satisfy this assumption (see~\cite[after Corollary 4.1]{Corr24}), as well as all the examples of isoperimetric profiles we mentioned above.

\smallskip 

An immediate consequence of Theorem~\ref{th:boundsForProfile intro} is the next statement.

\begin{corollaryletter}\label{cor:corB}
Let $p\geq 1$. Let $H$ be a finitely generated amenable group whose $\ell^p$-isoperimetric profile $\profp{H}$ satisfies Assumption~$(\star)$. Assume moreover that
\begin{equation*}
   \profp{H}\left(\frac{x}{\ln(x)}\right) \simeq \profp{H}(x)\text{ and }\profp{H}(x)\simeq\prof{H}(x).
\end{equation*}
Then one has 
\begin{equation*}
    \profp{\shuf{H}}(x) \simeq \profp{H}(\ln(x)).
\end{equation*}
\end{corollaryletter}

In practice, this result applies for many groups having slow enough profiles, for instance solvable Baumslag-Solitar groups $\text{BS}(1,n)$, lamplighters over $\Z^d$, or polycyclic groups with exponential growth. In particular, for the latter class, we recover~\cite[Corollary~6.9]{SCZ21}.

\begin{remark}
Corollary~\ref{cor:corB} implies that, if the $\ell^p$-isoperimetric profiles (for $p\ge 1$) of $H$ all have the same asymptotic behaviour, then the same holds for $\shuf{H}$, under mild assumptions on $H$.
\end{remark}

One class of groups for which estimates from~\cite{EZ21} and from~\cite{SCZ21} are not optimal is the one of iterated lampshufflers, defined inductively by $\shufn{n}{H}\defeq \shuf{\shufn{n-1}{H}}$ if $n\ge 1$ and $\shufn{0}{H}\defeq H$. It turns out that iterations of a finer version of Theorem~\ref{th:boundsForProfile intro} (see Corollary~\ref{cor:boundonProf} which leads to Corollary~\ref{cor:boundonProfIterated}) yield finer estimates, that we record in the two following statements. 

\begin{theoremletter}[see Theorem~\ref{th:profilesoflampjugglers}]\label{th:profilesoflampjugglers intro}
Let $p\geq 1$ and let $n\geq 0$ be an integer. Let $H$ be a finitely generated amenable group such that $\profp{H}(x)\simeq \prof{H}(x)$. Then the following holds for the $\ell^p$-isoperimetric profile of $\shufn{n}{H}$.
\begin{itemize}
    \item If $H$ has polynomial growth of degree $d\geq 1$, then
    \begin{equation*}
        \profp{\shufn{n}{s}{H}}(x)\simeq\left (\frac{\ln^{\circ n}(x)}{\ln^{\circ (n+1)}(x)}\right )^{\frac{1}{d}}.
    \end{equation*}
    \item If the $\ell^p$-isoperimetric profile $\profp{H}$ of $H$ satisfies Assumption~$(\star)$ and $\profp{H}\left(\frac{x}{\ln(x)}\right)\simeq \profp{H}(x)$, then
    \begin{equation*}
        \profp{\shufn{n}{H}}(x)\simeq\profp{H}(\ln^{\circ n}(x)).
    \end{equation*}
\end{itemize}
\end{theoremletter}

\paragraph{Halo products.} In~\cite{GT24a}, Genevois and Tessera introduced a general class of groups, called \textit{halo products}, as a natural generalisation of wreath products. This class encompasses lampshufflers, lampjugglers and lampcloners, and constitutes the suitable framework for our main result Theorem~\ref{th:recallThm intro}.

\begin{definition}
Let $X$ be a set. A \textit{halo of groups $\halo$ over $X$} is the data, for any subset $S\subset X$, of a group $L(S)$ such that:
\begin{itemize}
    \item for all $R,S\subset X$, if $R\subset S$ then $L(R)\leqslant L(S)$;
    \item $L(\emptyset)=\lbrace 1\rbrace$ and $L(X)=\langle L(S) : S\subset X \;\text{finite}\rangle$;
    \item for all $R,S\subset X$, $L(R\cap S)=L(R)\cap L(S)$.
\end{itemize}
\end{definition}

Given an action $H\curvearrowright X$ and a morphism $\alpha\colon H \longrightarrow \text{Aut}(L(X))$ satisfying $\alpha(h)(L(S))=L(hS)$ for any $S\subset X$ and $h\in H$, the \textit{permutational halo product} $\halo_{X,\alpha}X$ is the semi-direct product 
\begin{equation*}
    \halo_{X,\alpha}H \defeq L(X)\rtimes_{\alpha}H.
\end{equation*}

\sloppy In this paper, we focus on the case $X=H$. As mentioned, examples of halo products include
\begin{itemize}
    \item wreath products $F\wr H=(\bigoplus_H F)\rtimes H$, for which $L(S)=\bigoplus_S F$;
    \item lampshufflers $\shuf{H}=\fsym{H}\rtimes H$, for which $L(S)=\fsym{S}$,
\end{itemize}
and many other examples are introduced and studied in~\cite{GT24a}, such as
\begin{itemize}
    \item lampjugglers $\juggler{s}{H}=\fsym{H\times\lbrace 1,\dots,s\rbrace}\rtimes H$, with an integer $s\geq 1$, for which $L(S)=\fsym{S\times\lbrace 1,\dots,s\rbrace}$;
    \item lampcloners $\cloner{H}=\text{FGL}(H)\rtimes H$, with a field $\field$, for which $L(S)=\text{FGL}(S)$;
    \item lampdesigners $\designer{H}=(F\wr_{H}\fsym{H})\rtimes H$, with a non-trivial finite group $F$, for which $L(S)=F\wr_{S}\fsym{S}$, 
\end{itemize}
where $S$ denotes any subset of $H$. Here, $\text{FGL}(H)$ denotes the group of linear automorphisms of the abstract $\field$-vector space $V_{H}$ admitting $H$ as a basis, fixing all but finitely many basis vectors. We refer the reader to Section~\ref{sec:halo} for more details.

\smallskip

The motivation in~\cite{GT24a} to introduce such a general framework is that the semi-direct product structure provides a foliation of these spaces that must be, if $H$ satisfies additional mild assumptions, "quasi-preserved" by quasi-isometries, allowing the authors to show strong rigidity phenomena for quasi-isometries between such spaces, and thus extending the classification already obtained in~\cite{GT24b}. 

\paragraph{Isoperimetric profiles of halo products.} In this paper, our aim is to show that the halo structure is also particularly well-suited for tracking isoperimetric profiles of these groups. Namely, we prove the following two estimates on their F\o lner functions. The terminologies and notations are explained just after the statement.

\begin{theoremletter}[see Proposition~\ref{prop:upperboundonFol} and Theorem~\ref{th:lowerboundonFol}]\label{th:recallThm intro}
Let $p\ge 1$. Let $H$ be a finitely generated amenable group and let $S_{H}$ be a finite generating set of $H$. Let $\halo H$ be a naturally generated halo product over $H$.
\begin{enumerate}[label=(\roman*)]
    \item\label{item:1} If $\halo H$ is large-scale commutative and has finitely generated blocks, then for any $s_{0}\in S_{H}$, there exists a constant $C>0$ such that
\begin{equation*}
    \left(\fol{L(\lbrace1_{H},s_{0}\rbrace)}(x)\right)^{C\fol{H}(x)} \preccurlyeq \folp{\halo H}(x).
\end{equation*}
    \item 
If $\halo H$ has consistent blocks, then there exists a constant $C>0$ such that
\begin{equation*}
    \folp{\halo H}(x) \preccurlyeq \folp{H}(x)\cdot \Lambda_{\halo H}(C\cdot \folp{H}(x)).
\end{equation*}
\end{enumerate}
\end{theoremletter}

A halo product $\halo H$ has \textit{finite} (resp. \textit{finitely generated}) \textit{blocks} if $L(S)$ is finite (resp. finitely generated) for any finite subset $S\subset H$, and $\halo H$ has \textit{consistent} blocks if its blocks are finite and moreover the cardinality of $L(S)$ only depends on $|S|$. This assumption allows to define, as in~\cite{GT24a}, a function $\Lambda_{\halo H}\colon\N\longrightarrow\N$ sending any $n\in\N$ to $|L(S)|$ where $|S|=n$, called the \textit{lamp growth sequence} of $\halo H$.

\smallskip

Moreover, $\halo H$ is \textit{large-scale commutative} if there is $D\ge 0$ such that for any subsets $R,S\subset H$ that are at least $D$ far apart in $H$, $L(R)$ and $L(S)$ commute in $L(H)$. Such a notion has been introduced in~\cite{GT24a} as a key assumption to understand the general form of quasi-isometries between halo groups. Lastly, $\halo H$ is \textit{naturally generated} if it admits the natural and simplest generating set that we can imagine for a halo product, in view of the classical finite generating sets for lamplighters and lampshufflers.

\smallskip

For instance, a lamplighter $F\wr H$ and a lampshuffler $\shuf{H}$ are large-scale commutative (with $D=0$ for $F\wr H$, $D=1$ for $\shuf{H}$), are naturally generated, and have consistent blocks, with lamp growth sequences given by
\begin{equation*}
    \Lambda_{F\wr H}(n)=|F|^n \;\; \text{and} \;\; \Lambda_{\shuf{H}}(n)=n!.
\end{equation*}

\paragraph{Towards the proof of Theorem~\ref{th:recallThm intro}.} The lower bound is a direct computation, presented in Section~\ref{sec:UpperBoundFolner}, and inspired from the computations in the proof of~\cite[Theorem~6.8]{SCZ21} in the case of lampshufflers. We exhibit an explicit sequence of almost invariant functions $\halo H\longrightarrow \R$ from one such sequence of $H$. Namely, a sequence $(f_{n})_{n\in\N}$ of functions $H\longrightarrow\R$, realising the $\ell^p$-isoperimetric profile of $H$ (or equivalently its $\ell^p$-F\o lner function), gives rise to a sequence $(g_{n})_{n\in\N}$ for $\halo H$, defined by
\begin{equation*}
    \begin{array}{llcl}
    g_{n}\colon &\halo H &\longrightarrow &\R\\
    &(\sigma,h)&\longmapsto &f_{n}(h)\cdot \mathds{1}_{\sigma\in L(V_{n})}
    \end{array}
\end{equation*}
with $V_{n}\defeq\bigcup_{s\in S_{H}}{(\supp f_{n})s}$. This naturally provides a lower bound for the $\ell^p$-isoperimetric profile of $\halo H$.

\smallskip

The technical part is on the upper bound, and Section~\ref{sec:secondstrategy} provides such a bound in a general situation. In the particular case of lampshufflers, a strategy, well-known to the experts, consists in finding a "good" lamplighter subgroup of $\shuf{H}$, in the sense that this lamplighter should be based on a subgroup $K$ of $H$ which is quasi-isometric to $H$, or at least has the same isoperimetric profile. For this,~\cite[Proposition~2.4]{Sil24} is helpful. We may refer the reader to Appendix~\ref{appendixA} where the aforementioned method is presented and is instructive for the sequel. The upper bound then follows from the monotonicity of the $\ell^p$-isoperimetric profile when passing to finitely generated subgroups. Finding such lamplighter subgroups requires some algebraic assumptions on the base group, such as being non perfect or non co-Hopfian. Such classes of groups provide, at first glance, a nice framework (see Remark~\ref{rem:Preservation} and Proposition~\ref{prop:A7}) and encompass already many classical examples (e.g. all solvable groups).

\begin{remark}
    Notice that many halo products (lampjugglers $\juggler{s}{H}$ with $s\ge 2$, lampdesigners $\designer{H}$ and lampcloners $\cloner{H}$) still contain a lamplighter based on $H$ (we explain it in Section~\ref{sec:defHalo}, after the definition of lampcloners). On the other hand, there exist examples, other than lampshufflers, which do not contain a "good" lamplighter as a subgroup \textit{a priori}. In this paper, we exhibit such examples that we call \textit{lampupcloners}, which do not consider matrices (as lampcloners do) but upper triangular matrices with diagonal entries equal to $1$. The base group must be ordered for "upper triangular" to have a proper meaning and we will assume that the order is total. In this paper, we will focus on the case $H=\Z^d$ with the lexicographic order.
\end{remark}

The goal is to find another strategy for the lampshufflers or other halo products which do not contain a "good" lamplighter as a subgroup. Nonetheless, for these specific cases, the idea of finding substructures still remains fruitful. In Section~\ref{sec:secondstrategy}, we therefore make use of the more general notion of lamplighter graphs, that turn out to appear naturally in halo products as subgraphs. Large-scale commutativity will be a key ingredient since we need configurations of lamps to commute. In the particular case of lampshufflers, this novelty has the advantage, compared to Appendix~\ref{appendixA}, of requiring no assumptions on the base group $H$. We then conclude by establishing the monotonicity of the F\o lner function when passing to such subgraphs, in a similar manner as in~\cite[Lemma~4]{Ers03}.

\paragraph{Consequences of Theorem~\ref{th:recallThm intro}.} We first deduce from Theorem~\ref{th:recallThm intro} that, if $\halo H$ is large-scale commutative, is naturally generated and has consistent blocks, then for every $p\geq 1$, its $\ell^p$-F\o lner function satisfies
\begin{equation*}
    K^{\fol{H}(x)}\preccurlyeq\folp{\halo H}(x) \preccurlyeq \folp{H}(x)\cdot \Lambda_{\halo H}(C\cdot \folp{H}(x))
\end{equation*}
for some positive constants $C,K>0$.

\smallskip

Now, in terms of isoperimetric profiles, the main result is the following (see also Corollary~\ref{cor:boundonProfIterated} for the iterated version).

\begin{corollaryletter}[see Corollary~\ref{cor:boundonProf}]\label{cor:ProfHalo intro}
Let $H$ be a finitely generated amenable group. Let $\halo H$ be a naturally generated and large-scale commutative halo product having consistent blocks. Then, given $p\geq 1$, the $\ell^p$-isoperimetric profile of $\halo H$ satisfies
\begin{equation*}
    \frac{1}{C}\profp{H}\left (\frac{1}{C}\varphi^{-1}\left (x\right )\right )\leq\profp{\halo H}(x)\leq C\prof{H}(C\ln{x})
\end{equation*}
for some positive constant $C>0$, and with $\varphi(x)=x\cdot \Lambda_{\halo H}(x)$, where $\Lambda_{\halo H}$ is the lamp growth sequence of $\halo H$. 
\end{corollaryletter}

We then deduce Theorem~\ref{th:boundsForProfile intro} from this corollary. This result, combined with the fact that lampshufflers are subgroups of lampjugglers and lampdesigners, also implies that the isoperimetric profiles of the latters behave as the isoperimetric profiles of lampshufflers.

\begin{corollaryletter}[see Theorems~\ref{th:profilesoflampjugglers} and~\ref{th:profileoflampdesigners}]
Theorem~\ref{th:boundsForProfile intro} and Theorem~\ref{th:profilesoflampjugglers intro} also hold for lampjugglers and lampdesigners.
\end{corollaryletter}

Finally, let us also illustrate Corollary~\ref{cor:ProfHalo intro} with lampcloners and our new examples of halo products $\upcloner{\Z^d}$. Note that the estimates are less precise in the case of a base group of polynomial growth.

\begin{theoremletter}[see Theorem~\ref{th:profilesoflampcloners}]
Let $p\geq 1$ and let $n\geq 0$ be an integer. Let $H$ be a finitely generated amenable group such that $\profp{H}(x)\simeq \prof{H}(x)$. Let $\field$ be a field. Then the following holds for the $\ell^p$-isoperimetric profile of $\clonern{n}{H}$.
\begin{itemize}
    \item If $H$ has polynomial growth of degree $d\geq 1$, then
    \begin{equation}\label{eq:ClonerPolynomial}
        \left (\ln^{\circ n}(x)\right)^{\frac{1}{2d}}\preccurlyeq\profp{\clonern{n}{H}}(x)\preccurlyeq\left (\ln^{\circ n}(x)\right )^{\frac{1}{d}}.
    \end{equation}
    The same estimates hold for $\profp{\upcloner{\Z^d}}$, $d\geq 1$, where $\Z^d$ is endowed with the lexicographic order.
    \item If the $\ell^p$-isoperimetric profile $\profp{H}$ of $H$ satisfies Assumption~$(\star)$ and $\profp{H}\left(\sqrt{x}\right)\simeq \profp{H}(x)$, then
    \begin{equation*}
        \profp{\clonern{n}{H}}(x)\simeq\profp{H}(\ln^{\circ n}(x)).
    \end{equation*}
\end{itemize}
\end{theoremletter}

In the case of polynomial growth groups $H$, we have in fact the following slight improvement of~\eqref{eq:ClonerPolynomial} for the upper bound:
\begin{equation*}
    \profp{\cloner{H}}(x)\preccurlyeq\left(\frac{\ln(x)}{\ln(\ln(x))}\right)^{\frac{1}{d}},
\end{equation*}
since $\shuf{H}$ is a subgroup of $\cloner{H}$ (consider permutation matrices in $\mathrm{FGL}(H)$).

\paragraph{Applications to regular maps.} We now turn to the problem of the existence of quasi-isometries and regular maps between commonly studied spaces, which has been widely investigated in the literature, see e.g.~\cite{BST12, Tes20, HMT20, HMT22, Ben22, HMT25} among others. In these articles, the main guideline is, mostly, to associate to spaces new quantities that are monotonuous under regular maps, and that are thiner than the most obvious ones, such as volume growth or asymptotic dimension. As a concrete example, the volume growth does not say anything about the existence of a regular map 
\begin{equation*}
    \mathbb{H}_{\R}^{m_{1}}\times\R^{d_{1}} \longrightarrow \mathbb{H}_{\R}^{m_{2}}\times\R^{d_{2}}
\end{equation*}
whereas Poincaré profiles, introduced and studied in~\cite{HMT20}, impose a monotonic behaviour for the dimension of hyperbolic spaces and the growth exponent of the second factors~\cite[Corollary~1.14]{HMT22}.

\smallskip

On the amenable side, isoperimetric profiles remain powerful invariants to distinguish groups of exponential growth up to quasi-isometry. As an illustration:

\begin{theoremletter}[see Corollary~\ref{cor:IteratedShufflersPolynomialQIBiLip}]\label{th:IteratedShufflersPolynomialQIBiLip intro}
Let $n,m\ge 0$. Let $A$ and $B$ be infinite virtually abelian finitely generated groups, with growth degrees $a$ and $b$ respectively. Then the following are equivalent:
\begin{enumerate}[label=(\roman*)]
    \item $\shufn{n}{A}$ and $\shufn{m}{B}$ are quasi-isometric.
    \item $n=m$ and $a=b$.
    \item $\shufn{n}{A}$ and $\shufn{m}{B}$ are biLipschitz equivalent.
\end{enumerate}
\end{theoremletter}

\sloppy Two comments are in order here. Firstly, the fact that $\shufn{n}{A}$ and $\shufn{m}{B}$ are quasi-isometric implies that $a=b$ can also be detected with the asymptotic dimension. Indeed, $\shuf{A}$ (more generally $\shufn{n}{A}$) and $A$ have same asymptotic dimension. On the other hand, the asymptotic dimension does not detect numbers of iterations we make, whereas isoperimetric profiles do. These invariants are therefore more powerful with this respect. Additionally, regarding other monotonuous quantities under regular maps, volume growth is unhelpful, as it is exponential for both groups (when $n,m\geq 1$).

\smallskip

Secondly, it is worth noticing that, for (iterated) lampshufflers over virtually abelian groups or groups with slow profiles (see Corollary~\ref{cor:shufflersoverslowprofiles}), being quasi-isometric is the same as being biLipschitz equivalent. This rigidity is in sharp contrast with lamplighters over $\Z$~\cite{Dym10} or over one-ended groups~\cite{GT24b}, classes in which there are pairs of quasi-isometric groups that are not biLipschitz equivalent.

\smallskip

We refer the reader to Corollary~\ref{cor:IteratedShufflersPolynomialRegularMap} and Remark~\ref{rm:extensionsfornilpotentgroups} for asymmetric versions of Theorem~\ref{th:IteratedShufflersPolynomialQIBiLip intro}, about the existence of a regular map 
\begin{equation*}
    \shufn{n}{A}\longrightarrow \shufn{m}{B}
\end{equation*}
for polynomial growth groups $A$ and $B$ (not necessarily virtually abelian). A nice consequence of these studies is the following.

\begin{corollaryletter}[see Corollary~\ref{cor:QIRegularMap}]
Let $n,m\ge 0$. Let $A$ and $B$ be infinite virtually abelian finitely generated groups, with growth degrees $a$ and $b$ respectively. Then the following are equivalent:
\begin{enumerate}[label=(\roman*)]
    \item the three equivalent assertions of Theorem~\ref{th:IteratedShufflersPolynomialQIBiLip intro} hold;
    \item \sloppy $\shufn{n}{A}$ regularly embeds $\shufn{m}{B}$, and $\shufn{m}{B}$ regurlarly embeds into $\shufn{n}{A}$.
\end{enumerate}
\end{corollaryletter}

Another interesting consequence of our computations of isoperimetric profiles is the following statement, which cannot be reached with methods from~\cite{GT24a}, even for quasi-isometric or coarse embeddings. Indeed, in the latter is introduced a key property, called the \textit{thick bigon property}, which is a crucial assumption for the study of quasi-isometries between halo products. Unfortunately, this property is not stable under iterations of lampshufflers, and cannot be used for $\shufn{n}{\Z^d}$ for instance.

\begin{propositionletter}[see Corollary~\ref{cor:iteratedshufintoiteratedlamplighterPolynomial}]
Let $d,k,n\ge 1$ be three integers. If there exists a regular map 
\begin{equation*}
\shufn{n}{\Z^d}\longrightarrow \Z/2\Z\wr(\Z/2\Z\wr(\dots(\Z/2\Z\wr \Z^k))),
\end{equation*}
where the wreath product is iterated $n$ times, then $d<k$. 
\end{propositionletter}

In relation with the results from~\cite{GT24a}, we expect in fact that there is no regular map from $\shufn{n}{\Z^d}$ to $\Z/2\Z\wr(\Z/2\Z\wr(\dots(\Z/2\Z\wr \Z^k)))$, even when $d<k$. 

\smallskip

Finally, we emphasize here that similar results can be obtained for other halo products, such as lampjugglers, lampdesigners, and lampcloners (except when the base group has polynomial growth since we do not have a precise estimate of $\prof{\cloner{H}}$ in this case).

\paragraph{Plan of the paper.} After a few preliminaries in Section~\ref{sec:preliminaries}, we introduce halo products in Section~\ref{sec:halo}, with the main assumptions we will need to study them. Section~\ref{sec:computationsFolnerFunction} is devoted to the computation of F\o lner functions for halo products, and we deduce estimates for isoperimetric profiles in Section~\ref{sec:computationsIsoProf}. This finally implies existence and non-existence results of regular maps and quasi-isometries between such groups, see Section~\ref{sec:appQIandRegularMaps}. Lastly, Appendix~\ref{appendixA} presents various minimal algebraic assumptions under which lamplighters appear as subgroups of lampshufflers.

\paragraph{Acknowledgements.} We are grateful to our advisor Romain Tessera for his valuable guidance and support, and to Juan Paucar for many fruitful discussions. We are also thankful to Anthony Genevois, Jérémie Brieussel, Anna Erschler, and Eduardo Silva for their comments and suggestions on an earlier version of this paper.

\section{Notations and preliminaries}\label{sec:preliminaries}

\subsection{Notations}\label{sec:notations} 

Given non-decreasing functions $f,g\colon \R_{>0}\longrightarrow \R_{>0}$, we write $f(x)=O(g(x))$ if there exists $C>0$ such that $f(x)\le Cg(x)$ for all $x$ large enough, and $f(x)=o(g(x))$ if $\frac{f(x)}{g(x)}$ goes to $0$ as $x$ goes to $+\infty$. We write $f\sim g$, and we say that $f$ and $g$ are \textit{equivalent}, if $\frac{f(x)}{g(x)}$ goes to $1$ as $x$ goes to $+\infty$.  

\smallskip

A slightly weaker notion is the one of \textit{asymptotic equivalence}. Namely, if $f,g\colon \R_{+}\longrightarrow \R_{+}$ are non-decreasing, we say that $f$ is \textit{dominated} by $g$, and we write $f\preccurlyeq g$, if there is a constant $C>0$ such that $f(x)=O(g(Cx))$. We say that $f$ and $g$ are \textit{asymptotically equivalent}, written $f\simeq g$, if $f\preccurlyeq g$ and $g\preccurlyeq f$. Note that two equivalent functions are asymptotically equivalent.

\smallskip

Given a group $G$, we denote $1_{G}$ its neutral element, and if $G$ is generated by a finite set $S$, $\text{Cay}(G,S)$ refers to the Cayley graph of $G$ with respect to $S$, that is the graph whose vertices are elements of $G$ and whose edges are pairs of the form $(g,gs)$ with $g\in G$ and $s\in (S\cup S^{-1})\setminus\lbrace 1_{H}\rbrace$, while $|\cdot|_{S}$ denotes the usual length function on $G$ associated to $S$, and $d_{S}$ stands for the left-invariant word metric associated to $S$. The notation $V_{G,S}$ refers to the growth function of $G$, defined by $V_{G,S}(n)\defeq |\lbrace g\in G : |g|_{S}\le n\rbrace|$. Recall that its asymptotic behaviour, in the sense of $\simeq$, is independent of the choice of $S$.

\subsection{Coarse geometry and coarse maps} 

We now recall many basic concepts in the study of the large-scale geometry of metric spaces.

\smallskip

A map $f\colon (X,d_{X})\longrightarrow (Y,d_{Y})$ between two metric spaces is a \textit{coarse embedding} if there exist two functions $\sigma_{-},\sigma_{+}\colon [0,\infty)\longrightarrow [0,\infty)$ such that $\sigma_{-}(t)\longrightarrow\infty$ when $t\rightarrow\infty$ and such that 
\begin{equation*}
    \sigma_{-}(d_{X}(x,y)) \le d_{Y}(f(x),f(y)) \le \sigma_{+}(d_{X}(x,y))
\end{equation*}
for all $x,y\in X$. The maps $\sigma_{-}, \sigma_{+}$ are called the \textit{parameters} of $f$. If those parameters are both affine functions, we say that $f$ is a \textit{quasi-isometric embedding}, and without restrictions we may assume that $\sigma_{+}$ and $\sigma_{-}$ have multiplicative constants inverses of each other and that their additive constants differ by the sign. More precisely, given $C\ge 1$ and $K\ge 0$, we say that $f$ is a $(C,K)$-\textit{quasi-isometric embedding} if 
\begin{equation*}
    \frac{1}{C}\cdot d_{X}(x,y)-K \le d_{Y}(f(x),f(y)) \le C\cdot d_{X}(x,y)+K
\end{equation*}
for all $x,y\in X$. Additionally if $d_{Y}(y,f(X))\le K$ for all $y\in Y$, we say that $f$ is a \textit{quasi-isometry}, or a $(C,K)$-\textit{quasi-isometry}. When such a map exists, we say that $X$ and $Y$ are \textit{quasi-isometric}.

\smallskip

\sloppy For $C>0$, a $(C,0)$-quasi-isometry is usually called a \textit{biLipschitz equivalence} (note that those maps exactly coincide with bijective quasi-isometries), and a map $f\colon (X,d_{X})\longrightarrow (Y,d_{Y})$ satisfying only  
\begin{equation*}
    d_{Y}(f(x),f(y)) \le C\cdot d_{X}(x,y)
\end{equation*}
for any $x,y\in X$ is said to be \textit{$C$-Lipschitz}. Lastly, $f\colon X\longrightarrow Y$ is \textit{regular} if it is $C$-Lipschitz for some $C>0$ and pre-images of points have uniformly bounded cardinality: there is $m\ge 1$ such that 
$|f^{-1}(\lbrace y\rbrace)| \le m$ for any $y\in Y$. 

\smallskip

Note that any quasi-isometry is a quasi-isometric embedding, which is itself a coarse embedding, which is itself a regular map, but none of the reverse implications hold. For instance, the inclusion of a closed compactly generated subgroup in a locally compact compactly generated group is always a coarse embedding, while it is a quasi-isometry only if the subgroup is undistorted, and the map $\Z\longrightarrow\Z$, $n\longmapsto |n|$, is a regular map, while it is not a coarse embedding.

\subsection{Isoperimetric profiles} 

\paragraph{Isoperimetric profile for groups.} For a finitely generated group $G$ and a finite generating set $S_{G}$, its $\ell^{p}$-\textit{isoperimetric profile}, for $p\geq 1$, is the function $\profp{G}\colon \N\longrightarrow \R_{+}$ given by
\begin{equation*}
    \profp{G}(n)\defeq\sup_{\substack{f\colon G\to\R_+\\|\supp{f}|\leq n}}{\frac{\|f\|_{p}}{\|\nabla f\|_{p}}}
\end{equation*}
where the support of $f\colon G\longrightarrow\R_+$ is $\supp{f}\defeq\lbrace g\in G : f(g)\not=0\rbrace$ and the $\ell^p$-norm of its gradient is defined by $\|\nabla f\|_{p}^{p}\defeq\sum_{g\in G,\; s\in S_{G}}{|f(g)-f(gs)|^p}$. For $p=1$, the $\ell^1$-isoperimetric profile is simply called \textit{isoperimetric profile} and one has
\begin{equation*}
    \prof{G}(n) \simeq \sup_{|A|\le n}\frac{|A|}{|\partial_{G} A|}
\end{equation*}
where $\partial_{G}A \defeq AS_{G}\setminus A=\lbrace g\in G\setminus A : \exists s\in S_{G}, \exists h\in A, g=hs\rbrace$ is the \textit{boundary} of $A$ in $G$.

\smallskip

Recall also that the isoperimetric profile of a group $G$ is the generalised inverse of its \textit{F\o lner function} $\fol{G}\colon \N\longrightarrow\R_{+}$, defined as
 \begin{equation*}
     \fol{G}(n) \defeq \inf\left\lbrace |A| : \frac{|\partial_{G}A|}{|A|} \le \frac{1}{n} \right\rbrace.
 \end{equation*}

We more generally define the $\ell^p$-F\o lner function $\folp{G}\colon\N\longrightarrow \R_{+}$ for every $p\ge 1$, as 
\begin{equation*}
     \folp{G}(n) \defeq \inf\left\lbrace \left|\supp{f}\right| : \frac{\|\nabla f\|_p}{\|f\|_p} \le \frac{1}{n} \right\rbrace.
\end{equation*}
For every $p\geq 1$, $\folp{G}$ and $\profp{G}$ are generalised inverses of each other, and we have $\textup{F\o l}_{1,G}(x)\simeq\fol{G}(x)$. Thus, in the sequel, we will always write $\fol{G}$ instead of $\textup{F\o l}_{1,G}$.

\begin{remark}\label{rem:InverseEachOther}
Notice that, given the asymptotic behaviour of the $\ell^p$-F\o lner function, we can deduce the asymptotic behaviour of the $\ell^p$-isoperimetric profile, even though they are not real inverses of each other, but only generalised inverses \textit{a priori}. Indeed, a non-decreasing function $\R_{+}\longrightarrow \R_{+}$ is always asymptotically equivalent to an increasing function $\R_{+}\longrightarrow \R_{+}$ (cf.~\cite[Remark~1.2]{Corr24}) and it is not hard to check that $\simeq$ is preserved when passing to generalised inverses. Thus, in the sequel, we can and will assume that the $\ell^p$-F\o lner function and the $\ell^p$-isoperimetric profile are injective and then real inverses of each other. Hence, studying the asymptotic behaviour of $\profp{G}$ is the same as studying the asymptotic behaviour of $\folp{G}$. 
\end{remark}

The asymptotic behaviour of isoperimetric profiles is stable under quasi-isometries, and is in particular independent of the choice of a generating set for $G$. More generally, $\ell^p$-isoperimetric profiles are monotonuous when passing to finitely generated subgroups:

\begin{theorem}[{\cite[Lemma~4]{Ers03}}]\label{th:profilemonotonuousSubgroup}
Let $H$ be a finitely generated subgroup of a finitely generated group $G$. Then one has $\profp{G}(n) \preccurlyeq \profp{H}(n)$ for every $p\ge 1$.
\end{theorem}

This has been widely generalised to general regular maps by Delabie, Koivisto, Le Maître and Tessera, using connections with quantitative measure equivalence.

\begin{theorem}[{\cite[Theorem~1.5]{DKLMT22}}]\label{th:profilemonotonuous}
Let $G$ and $H$ be finitely generated amenable groups. If there exists a regular map from $G$ to $H$, then $\profp{H}(n) \preccurlyeq \profp{G}(n)$ for every $p\ge 1$.
\end{theorem}

Equivalently, both results can be stated in terms of F\o lner functions.

\smallskip

Note that a group is amenable if and only if its isoperimetric profile is unbounded. The idea to keep in mind is that the isoperimetric profile is a measurement of how much amenable a group is. The faster the isoperimetric profile tends to infinity, the more the group is amenable. In particular, the isoperimetric profile is a particularly well suited invariant to distinguish amenable groups with exponential growth up to quasi-isometries or regular maps, and it has now been computed for many classes of groups, among which:
\begin{itemize}
    \item $\profp{G}(n) \simeq n^{\frac{1}{d}}$ if $G$ has polynomial growth of degree $d\ge 1$;
    \item $\profp{G}(n)\simeq \ln(n)$ for solvable Baumslag-Solitar groups and lamplighters $F\wr\Z$, where $F$ is a non-trivial finite group;
    \item $\profp{G}(n) \simeq \ln(n)$ for any polycyclic group with exponential growth~\cite{Pit95, Pit00}, or more generally any exponential growth group within the class GES of Tessera~\cite[Corollary 5]{Tes13};
    \item $\profp{F\wr N}(n)\simeq \ln(n)^{\frac{1}{d}}$ with $F$ finite and non-trivial, and $N$ of growth degree $d\ge 1$~\cite{Ers03}.
\end{itemize}

More generally,~\cite[Theorem 1]{Ers03} provides a general formula for computing F\o lner functions of wreath products, so for instance the last example can be extended to iterated wreath products. Let us explain with more details Erschler's result, since it will play an important role in the sequel.

\smallskip

Let $G$ and $H$ be finitely generated groups and assume that for every $C>0$, there exists $k>0$ such that $\fol{H}(kn)>C\cdot \fol{H}(n)$ for any large enough integer $n$. Then we have
\begin{equation*}
    \fol{G\wr H}(x)\simeq\left(\fol{G}(x)\right)^{\fol{H}(x)}.
\end{equation*}
This assumption on $\fol{H}$, introduced in~\cite{Ers03}, can also be stated in terms of isoperimetric profile in the following way: for every $C>0$, we have $\prof{H}(Cx)=O\left(\prof{H}(x)\right)$. This assumption also appeared in~\cite{Corr24} and we do not have any example of a finitely generated group for which it does not hold. In Section~\ref{sec:computationsIsoProf}, we will make use of this mild assumption, and we will refer to it as \textit{Assumption}~($\star$). In~\cite{Ers03}, this assumption is used to get rid of some constant appearing in the lower bound: there exists $C>0$ such that
\begin{equation}\label{eq:ConstantAssumptionStar}
    \fol{G\wr H}(x)\succcurlyeq\left (\fol{G}(x)\right )^{C\fol{H}(x)};
\end{equation}
whereas the upper bound is exactly $\fol{G\wr H}(x)\preccurlyeq\left (\fol{G}(x)\right )^{\fol{H}(x)}$. In the particular case of a finite group $G$, we have $\prof{G\wr H}(x)\simeq\prof{H}(\ln(x))$ if $\prof{H}$ satisfies Assumption~($\star)$.

\smallskip

The lower bound~\eqref{eq:ConstantAssumptionStar} also holds in the context of lamplighter graphs, see Section~\ref{sec:secondstrategy}. To relate the work of Erschler on the $\ell^1$-F\o lner function with the $\ell^p$-F\o lner function (for $p\geq 1$) that we want to compute for halo products, we will first have to reduce the proof to the case $p=1$, thanks to the well-known fact that if $p\geq 1$, then the $\ell^p$-F\o lner function dominates the $\ell^1$-F\o lner function (see Lemma~\ref{lem:MonotonuousProfile}). In fact, it is conjectured that the $\ell^p$-F\o lner functions, for $p\ge 1$, all have the same asymptotic behaviour.

\medskip

A fundamental result in geometric group theory is the one of Coulhon and Saloff-Coste, who proved in~\cite{CSC93} that for a finitely generated group $G$ and a finite symmetric generating set $S$ of $G$, one has 
\begin{equation*}
    \frac{|\partial_{G}F|}{|F|} \ge \frac{1}{4|S|}\cdot\frac{1}{\Phi_{G}(2|F|)}
\end{equation*}
for any finite set $F\subset G$, where $\Phi_{G}\colon \R_{>0} \longrightarrow \N$, $\Phi_{G,S}(t) \defeq \min\lbrace n \ge 0 : V_{G,S}(n)>t\rbrace$ is the \textit{inverse growth function} of $G$. Since then, it has constantly been improved to thiner inequalities, see for instance~\cite{PS22}. Inverting this inequality and taking the sup, one directly gets the upper bound 
\begin{equation*}
    \prof{G}(n) \preccurlyeq \Phi_{G,S}(n)
\end{equation*}
on the $\ell^{1}$-isoperimetric profile of $G$. This upper bound is optimal when $G$ has polynomial growth, while if it has exponential growth, one only gets $\prof{G}(n) \preccurlyeq \ln(n)$. In fact,~\cite[Theorem~1.1]{BZ21} describes a large class of possible asymptotic behaviours for isoperimetric profiles of finitely generated groups with exponential growth, namely for any non-decreasing function $f$ such that $x\longmapsto \frac{x}{f(x)}$ is non-decreasing, there exists a finitely generated group of exponential growth whose $\ell^p$-isoperimetric profile is $\simeq \frac{\ln(x)}{f(\ln(x))}$ for every $p\geq 1$.

\smallskip

Isoperimetric profiles are also particularly studied for their relations with return probabilities of random walks on groups, see for instance~\cite{SCZ15, SCZ16, SCZ18, BZ21}.

\paragraph{Isoperimetric profile for graphs.} Isoperimetric profiles can be defined, more generally, in the framework of bounded degree graphs, without any underlying algebraic structure. In this paper, we only focus on the case $p=1$, but similar definitions can be made for $p>1$.

\smallskip

Since there will be no ambiguity, we will abusively use the same notation for a graph and the set of its vertices. Given a graph $Y$, the presence of an edge between two vertices $v$ and $w$ will be denoted by $v\sim_{Y}w$.

\smallskip

Given a graph $Y$, its isoperimetric profile is the function $\prof{Y}\colon\N\longrightarrow\R_{+}$ defined by
\begin{equation*}
    \prof{Y}(n)\defeq\sup_{|A|\le n}{\frac{|A|}{|\partial_{Y} A|}}
\end{equation*}
where, given a finite set $A\subset Y$ of vertices, $\partial_{Y}A\defeq\lbrace v\in Y\setminus A : \exists a\in A, v\sim_{Y} a\rbrace$ is the \textit{boundary} of $A$ in the graph $Y$.

\smallskip

In the case of Cayley graphs of finitely generated groups, we recover the corresponding notion of isoperimetric profile of groups, defined above. Moreover, the invariance of isoperimetric profile under quasi-isometry is still valid in this more general setup. Finally, we analogously define the F\o lner function of a graph.

\smallskip

This setup of graphs will be crucial in our paper. Indeed, in Section~\ref{sec:secondstrategy}, we will define lamplighter graphs and will require a lower bound of their F\o lner functions.

\section{Halo products}\label{sec:halo}

In this section, we define halo products and the main classes we are interested in.

\subsection{Halo products: definition and main examples}\label{sec:defHalo} 

\begin{definition}
Let $X$ be a set. A~\textit{halo of groups $\halo$ over $X$} is the data, for any subset $S\subset X$, of a group $L(S)$ such that:
\begin{itemize}
    \item for all $R,S\subset X$, if $R\subset S$ then $L(R)\leqslant L(S)$;
    \item $L(\emptyset)=\lbrace 1\rbrace$ and $L(X)=\langle L(S) : S\subset X \;\text{finite}\rangle$;
    \item for all $R,S\subset X$, $L(R\cap S)=L(R)\cap L(S)$.
\end{itemize}
\end{definition}

Given an action $H\curvearrowright X$ and a morphism $\alpha\colon H \longrightarrow \text{Aut}(L(X))$ satisfying $\alpha(h)(L(S))=L(hS)$ for any $S\subset X$ and $h\in H$, the~\textit{permutational halo product} $\halo_{X,\alpha}H$ is the semi-direct product 
\begin{equation*}
    \halo_{X,\alpha}H \defeq L(X)\rtimes_{\alpha}H.
\end{equation*}

The definition is motivated by permutational wreath products, which are basic examples of permutational halo products. Indeed, given groups $F,H$ and an action $H\curvearrowright X$, set $L(S)\defeq \bigoplus_{S}F$ for any $S\subset X$. Then $\halo_{X,\alpha}H$ coincides with $F\wr_{X}H$, where $\alpha$ is the action of $H$ on $\bigoplus_{X}F$ obtained by permuting the coordinates through the initial action $H\curvearrowright X$. In particular, for $X=H$ and the left-multiplication action of $H$ on itself, we recover a description of the wreath product $F\wr H$ as a halo product. 

\smallskip

Let us now describe other examples of halo products. From now on, we only focus on halo products with $X=H$, that we simply denote by $\halo H$, for the natural action of $H$ on itself by left-multiplication.

\paragraph{Lampshufflers.} Let $H$ be a group, and let $\fsym{H}$ be the group of~\textit{finitely supported} permutations of $H$, that is the group of bijections $H\longrightarrow H$ that are the identity outside a finite subset of $H$. The group $H$ acts naturally on $\fsym{H}$, via 
\begin{equation*}
    (h\cdot\sigma)(x) \defeq h\sigma(h^{-1}x), \; x\in H
\end{equation*}
for any $h\in H$ and $\sigma\in\fsym{H}$. Indeed, if $\sigma\colon H\longrightarrow H$ is a finitely supported bijection and $h\in H$, then so is $h\cdot \sigma$ and $\supp{(h\cdot\sigma)}=h\cdot\supp{(\sigma)}$, where $\supp{(\sigma)} \defeq\lbrace x\in H : \sigma(x) \neq x\rbrace$.

\smallskip

The~\textit{lampshuffler group over $H$}, denoted $\shuf{H}$, is then defined as the semidirect product 
\begin{equation*}
    \shuf{H}\defeq \fsym{H} \rtimes H.
\end{equation*}

It coincides with the halo product $\halo H$ where $L(S)\defeq \fsym{S}$ for any $S\subset H$. Additionally, if $H$ is finitely generated and $S_{H}$ denotes a finite generating set, then $\shuf{H}$ is generated by the finite set 
\begin{equation*}
    \Sigma_{H}\defeq\left\lbrace(\tau_{1_{H},s},1_{H}) : s\in S_{H}\right \rbrace\cup\left\lbrace(\text{id},s) : s\in S_{H}\right\rbrace
\end{equation*}
where, given any $x,y\in H$, $\tau_{x,y}\in\fsym{H}$ is the transposition that swaps $x$ and $y$, that is $\tau_{x,y}(x)=y$, $\tau_{x,y}(y)=x$ and $\tau_{x,y}(h)=h$ for any $h\neq x,y$. 

An element $(\sigma,h)\in\shuf{H}$ can be seen as a labelling of the vertices of the Cayley graph $\text{Cay}(H,S_{H})$ (a vertex $p\in H$ carries the label $\sigma(p)$) together with an arrow pointing at some vertex $h\in H$, and there are two types of moves in $\text{Cay}(\shuf{H},\Sigma_{H})$ to go from $(\sigma, h)$ to a neighbouring vertex:
\begin{itemize}
    \item either the arrow goes from $h$ to a neighbouring vertex in $H$;
    \item or the arrow stands on the vertex $h\in H$, and swaps its label with the label of one of its neighbours in $H$.
\end{itemize} 

\paragraph{Lampjugglers.} Lampshufflers are in fact particular instances of a broader family of groups, called~\textit{lampjugglers}. Given a group $H$ and an integer $r\ge 1$, the~\textit{lampjuggler over $H$} is the semi-direct product
\begin{equation*}
    \juggler{r}{H} \defeq \fsym{H\times\lbrace 1,\dots,r\rbrace} \rtimes H
\end{equation*}
where $H$ acts on $\fsym{H\times\lbrace 1,\dots,r\rbrace}$ through its initial action on $H\times\lbrace 1,\dots,r\rbrace$ given by $h\cdot (x,i) \defeq (hx, i)$. It can be described as the halo group $\halo H$ where $L(S)\defeq \fsym{S\times\lbrace 1,\dots, r\rbrace}$, $S\subset H$. As for lampshufflers, lampjugglers over finitely generated groups are finitely generated, and one can check that if $S_{H}$ is a finite generating set for $H$, then the finite set
\begin{equation*}
    \lbrace (\tau_{(1_{H},i),(s,j)}, 1_{H}) : s\in S_{H}, 1\le i,j\le r\rbrace  \cup \lbrace (\text{id},s) : s\in S_{H}\rbrace
\end{equation*}
generates $\juggler{r}{H}$. Here, an element $(\sigma, h)\in\juggler{r}{H}$ can be seen as a labelling of the vertices of $\text{Cay}(H,S_{H})\times\lbrace 1,\dots,r\rbrace$ together with an arrow pointing at some vertex $h\in H$. Right-multiplying $(\sigma, h)$ by a generator from the above set amounts either to move the arrow from $h$ to a neighbouring vertex $hs$ in $H$, or to keep the arrow on $h\in H$ and switching the labels of two vertices in $h \times\lbrace 1, \dots, r\rbrace$ and $hs \times \lbrace1, \dots, r\rbrace$ for some neighbour $hs$ of $h$.

\paragraph{Lampdesigners.} Let $F$ and $H$ be two groups. The~\textit{lampdesigner over $H$} is the semi-direct product 
\begin{equation*}
    \designer{H} \defeq (F\wr_{H}\fsym{H})\rtimes H
\end{equation*}
where $H$ acts on $\bigoplus_{H}F$ by permuting the coordinates through its action on itself by left-multiplication and acts on $\fsym{H}$ as described above. It is the halo product $\halo H$ for the collection $L(S)\defeq F\wr_{S}\fsym{S}$, $S\subset H$. 

Lampdesigners are close from lampjuggler groups, and in fact if $F$ is finite, $\designer{H}$ is a subgroup of $\juggler{|F|}{H}$, via the map 
\begin{align*}
    \begin{array}{cll}
    \designer{H} &\longrightarrow &\juggler{|F|}{H} \\
    ((f,\sigma),h) &\longmapsto &(\sigma', h)
    \end{array}
\end{align*}
where, given a pair $(f,\sigma)\in F\wr_{H}\fsym{H}$, $\sigma'$ is the permutation of $H\times F$ given by $\sigma'(h,i)=(\sigma(h), f(h)i)$. Note also that $\designer{H}$ contains $\shuf{H}$ as a subgroup. 

\paragraph{Lampcloners.} Let $H$ be a group and let $\field$ be a field. Denote $V_{H}$ the $\field$-vector space admitting $H$ as a basis, and denote $\lbrace e_{u} : u\in H\rbrace$ a formal basis. Let $\text{FGL}(H)$ be the group of linear automorphisms $V_{H}\longrightarrow V_{H}$ that fix all but finitely many basis elements. This group can also be seen as the group of finitely supported invertibles matrices with coefficients in $\field$ whose entries are indexed by $H\times H$. Once again, the action of $H$ on itself naturally yields an action of $H$ on $\text{FGL}(H)$. The \textit{lampcloner over $H$} is the semi-direct product 
\begin{equation*}
    \cloner{H} \defeq \text{FGL}(H)\rtimes H.
\end{equation*}
It is a halo product, for the collection $L(S)\defeq \text{FGL}(S)$, for every $S\subset H$, where $\text{FGL}(S)$ is thought of as the subgroup of $\text{FGL}(H)$ of linear automorphisms $V_{H}\longrightarrow V_{H}$ that fix $H\setminus S$ and that stabilise the subspace $\langle S\rangle\subset V_{H}$. 

In addition, if $\field$ is finite and if $H=\langle S_{H}\rangle$ is finitely generated, then the finite set
\begin{equation*}
    \lbrace (\delta_{1_{H}}(\lambda), 1_{H}) : \lambda \in \field\setminus\lbrace 0\rbrace\rbrace\cup\lbrace (\tau_{1_{H},s}(\lambda),1_{H}) : s\in S_{H}, \lambda \in \field\setminus\lbrace 0\rbrace\rbrace \cup\lbrace (\text{id}_{V_H}, s) : s\in S_{H}\rbrace
\end{equation*}
generates $\cloner{H}$, where, given $p,q\in H$ and $\lambda\in \field\setminus\lbrace 0\rbrace$, $\delta_{p}(\lambda)$ is the~\textit{diagonal matrix} 
\begin{align*}
    \delta_{p}(\lambda)\colon
    \begin{array}{cll}
    V_{H}&\longrightarrow & V_{H}\\
    \displaystyle \sum_{h\in H}\mu_{h}e_{h} &\longmapsto &\displaystyle \sum_{h\neq p}\mu_{h}e_{h}+\lambda\mu_{p}e_{p}
    \end{array}
\end{align*}
and $\tau_{pq}(\lambda)$ is the~\textit{transvection} 
\begin{align*}
    \tau_{pq}(\lambda)\colon
    \begin{array}{cll}
    V_{H}&\longrightarrow &V_{H} \\
    \displaystyle \sum_{h\in H}\mu_{h}e_{h} &\longmapsto &\displaystyle \sum_{h\neq p}\mu_{h}e_{h}+(\mu_{p}+\lambda\mu_{q})e_{p}
    \end{array}.
\end{align*}
Thus, thinking of an element $(\varphi,p)\in\cloner{H}$ as a labelling of $\text{Cay}(H,S_{H})$ (the vertex $h\in H$ has the label $\varphi(e_{h})\in V_{H}$), together with an arrow pointing at $p\in H$, right multiplying $(\varphi,p)$ by a generator from the above set amounts either to move the arrow to an adjacent vertex $q$ of $p$ in $H$; or to keep the arrow where it stands and multiply $\varphi(e_{p})$ by a non-trivial element of $\field$; or to keep the arrow where it stands and to~\textit{clone} the label $\varphi(e_{p})$ and add it to the label of a neighbour of $p$ after multiplication by an element of $\field\setminus\lbrace 0\rbrace$. 

\medskip

We refer the reader to~\cite[Section~2]{GT24a} for many other possible constructions, such as lampbraiders and verbal halo products, that encompass for instance nilpotent and metabelian wreath products.

\smallskip 

As mentioned in the introduction, the challenging part for the computation of isoperimetric profiles is to find the optimal upper bound. This is done using the monotonicity of the isoperimetric profiles when passing to suitable substructures such as subgroups or even subgraphs. Using "good" subgroups is a technique known widely by the experts in the case of lampshufflers, but requires some algebraic assumptions on the base groups (see Appendix~\ref{appendixA} for more details). It is now a good place to observe that the other examples of halo products we give above already have lamplighters as subgroups, with the same base groups. More precisely:
\begin{itemize}
    \item For a lampjuggler $\juggler{s}{H}$, with $s\ge 2$, we consider the subgroup 
    \begin{equation*}
    G\defeq \lbrace (\sigma, h)\in\juggler{s}{H} : \sigma(\lbrace k\rbrace\times F)=\lbrace k\rbrace\times F \;\text{for all $k\in H$}\rbrace.
    \end{equation*}
    One can check directly that $G$ is isomorphic to $\sym{\lbrace 1,\dots,r\rbrace}\wr H$.
    \item For a lampdesigner $\designer{H}$, we notice that $\bigoplus_{H}{F}$ is a subgroup of $F\wr_{H}\sym{H}$, invariant under the natural action of $H$ on $F\wr_{H}\sym{H}$, so $F\wr H$ is a subgroup of $\designer{H}$.
    \item Considering the subgroup of $\mathrm{FGL}(H)$ generated by diagonal matrices $\delta_{h}(\lambda)$ for $h\in H$ and $\lambda\in\field\setminus\lbrace 0\rbrace$, we easily prove that $(\field\setminus\lbrace 0\rbrace)\wr H$ is a subgroup of $\cloner{H}$.
\end{itemize}

Thus, for these examples, techniques from Section~\ref{sec:computationsFolnerFunction} are unnecessary. Let us then give new examples of halo products which do not contain "good" lamplighter subgroups \textit{a priori}.

\paragraph{Lampupcloners.} Let $H$ be a totally ordered group. Let $\field$ be a field. Denote $V_{H}$ the $\field$-vector space admitting $H$ as a basis, and denote $\lbrace e_{u} : u\in H\rbrace$ a formal basis. Let $\text{FU}(H)$ be the subgroup of $\text{FGL}(H)$ generated by the transvections $\tau_{p,q}(\lambda)$ for $p<q$ and $\lambda\in\field$. This group can also be seen as the group of finitely supported upper triangular matrices with coefficients in $\field$, whose entries are indexed by $H\times H$, and with diagonal entries equal to $1$. Once again, the action of $H$ on itself naturally yields an action of $H$ on $\text{FU}(H)$. The~\textit{lampupcloner over $H$} is the semi-direct product 
\begin{equation*}
    \upcloner{H} \defeq \text{FU}(H)\rtimes H.
\end{equation*}
It is a halo product, for the collection $L(S)\defeq \text{FU}(S)$, for every $S\subset H$, where $\text{FU}(S)$ is thought of as the subgroup of $\text{FU}(H)$ of linear automorphisms $V_{H}\longrightarrow V_{H}$ that fix $H\setminus S$ and that stabilise the subspace $\langle S\rangle\subset V_{H}$. Indeed, to prove the property with the intersection in the definition of halo product, we notice that $\text{FU}(S)$ is the set of linear automorphisms $\varphi\in\text{FGL}(S)$ satisfying $(e_{q})_{*}(\varphi(e_{p}))=0$ for every $p,q\in H$ satisfying $p<q$ (where $((e_{h})_{*})_{h\in H}$ is the family of coordinate functions for the basis $(e_{h})_{h\in H}$ of $V_H$). Notice that here we need crucially a total order on $H$.

\smallskip

We do not know if the finite set
\begin{equation*}
    \lbrace (\tau_{1_{H},s}(\lambda),1_{H}) : s\in S_{H}\rbrace \cup\lbrace (\text{id}_{V_H}, s) : s\in S_{H}\rbrace
\end{equation*}
always generates $\upcloner{H}$, where $S_H$ is a generating subset of $H$ satisfying $s\geq 1_H$ for every $s\in S_H$. In Proposition~\ref{prop:UpclonerFG}, we prove that this is the case for $H=\Z^d$, endowed with the lexicographic order.

\subsection{Important assumptions}\label{sec:ImportantAssumption} 

Our main results deal with halo products satisfying various important assumptions that we introduce in this section.

\subsubsection{Large-scale commutativity} 

The first one has been introduced in~\cite{GT24a}, under the terminology~\textit{large-scale commutativity}.
\begin{definition}
Let $\halo H$ be a halo product over a finitely generated group $H$, and let $S_{H}$ be a finite generating set of $H$. We say that $\halo H$ is \textit{large-scale commutative} if there exists a constant $D\ge 0$ such that, for any $R,S\subset H$ with $d_{S_{H}}(R,S)\ge D$, the subgroups $L(R)$ and $L(S)$ commute in $L(H)$.
\end{definition}

This notion plays a key role in the quasi-isometry classification of halo groups established in~\cite{GT24a}, see e.g.~\cite[Theorem~6.3 and Theorem~6.6]{GT24a}. Examples of large-scale commutative halo products include lamplighters ($D=0$), lampshufflers ($D=1$), lampcloners ($D=1$) and lampupcloners ($D=1$).

\subsubsection{Finite generating sets} 

Let us now turn to terminologies more specific to halo groups over finitely generated groups. Inspired by lampshufflers, when we are looking for a generating set of a general halo product, there is a natural candidate. The~\textit{natural generation property}, that we now introduce, is by definition satisfied by a halo product having
this natural candidate as generating set.

\begin{definition}
Let $H$ be a finitely generated group, with a finite generating subset $S_{H}$. We say that a halo product $\halo H$ over $H$ is~\textit{naturally generated} if it is generated by the set
\begin{equation*}
\lbrace (1_{L(H)},s) : s\in S_{H}\rbrace \cup \bigcup_{s\in S_{H}}\lbrace (\sigma, 1_{H})\in\halo H : \sigma\in L(\lbrace 1_{H},s\rbrace)\rbrace.
\end{equation*}
\end{definition}

We already know from the previous section that our running examples, except lampupcloners, are naturally generated. In this section, we actually prove it once more, highlighting a more general phenomenon. We will also prove that lampupcloners over free abelian groups (endowed with the lexicographic order) are naturally generated.

\smallskip 

Let us now introduce a terminology relative to the generation for blocks of a halo product. 

\begin{definition}
We say that a halo product $\halo H$ has~\textit{finite} (resp.~\textit{finitely generated}) blocks if, for any finite subset $S\subset H$, $L(S)$ is finite (resp. finitely generated).
\end{definition}

For instance, a wreath product $F\wr H$, where $F$ is finitely generated, has finitely generated blocks. Moreover, lampjugglers, lampdesigners, lampcloners and lampupcloners have finite blocks, so they have finitely generated blocks.

\smallskip 

If a naturally generated halo product has finitely generated blocks, then it has a natural finite generating set:

\begin{fact}\label{fact:BlocksHaloFG}
Let $H$ be a finitely generated group and let $S_{H}$ be a finite symmetric generating set of $H$. Let $\halo H$ be a halo product over $H$. Suppose that $\halo H$ is naturally generated and has finitely generated blocks. Then the finite set
\begin{equation*}
    S_{\halo H}\defeq\lbrace (1_{L(H)},s) : s\in S_{H}\rbrace \cup \bigcup_{s\in S_{H}}\lbrace (\sigma, 1_{H})\in\halo H : \sigma\in S(s)\rbrace
\end{equation*}
generates $\halo H$, where $S(s)$ is any finite generating subset of $L(\lbrace 1_{H},s\rbrace)$.\qed
\end{fact}

Natural generation property turns out to be equivalent to another property, that we call the~\textit{decreasing length property} and that we now define. This equivalent definition will lead us to a proof by induction when we will need to check that a given halo product is naturally generated.

\begin{definition}
Let $H$ be a finitely generated group, with finite generating set $S_{H}$, and let $\halo H$ be a halo product over $H$. Given a finite subset $R\subset H$, we define its~\textit{length} as
\begin{equation*}
        |R|_{H}\defeq\sum_{h\in R}{|h|_{H}}.
\end{equation*}
We say that $\halo H$ has the~\textit{decreasing length property} if for every finite subset $R\subset H$ such that $|R|_H\geq 2$, there exist a positive integer $k\ge 1$ and $k$ subsets $R_{1},\dots, R_{k}$ of $H$ such that
\begin{itemize}
    \item for every $i\in\lbrace 1,\dots,n\rbrace$, one has $|R_{i}|_{H}<|R|_{H}$;
    \item $L(R)\le \langle hL(R_{i}) : h\in H, 1\le i\le n\rangle$.
\end{itemize}
\end{definition}

Here is the proof of the claimed equivalence. 

\begin{proposition}\label{prop:DefEquivalentFG}
Let $H$ be a finitely generated group and let $S_{H}$ be a finite generating set of $H$. Let $\halo H$ be a halo product over $H$. The following assertions are equivalent:
\begin{enumerate}[label=(\roman*)]
        \item $\halo H$ is naturally generated.
        \item $\halo H$ has the decreasing length property.
\end{enumerate}
\end{proposition}

\begin{proof}
Assume that $\halo H$ is naturally generated. Let $R$ be a subset of $H$ such that $|R|_{H}\ge 2$. Given $\sigma\in L(R)$, we know that $(\sigma,1_{H})$ can be written as a product
\begin{equation*}
        (\sigma_{1},h_{1})\dots (\sigma_{n},h_{n})
\end{equation*}
with $h_{1},\dots,h_{n}\in H$ and $\sigma_{1}\dots,\sigma_{n}\in\bigcup_{s\in S_{H}}{L(\lbrace 1_{H},s\rbrace)}$ for every $i\in\lbrace 1,\dots,n\rbrace$. The composition law of the halo product implies that $\sigma$ is a product of elements of the form $h_1\ldots h_{i-1}\sigma_{i}$ for $i\in\lbrace 1,\dots,n\rbrace$. Choosing $R_{1},R_{2},\dots$ as sets $\{1_H,s\}$ for $s\in S_{H}$, we have proved the decreasing length property with respect to $R$.

\smallskip 

Let us now assume that $\halo H$ has the decreasing length property. We first claim that:
    
\begin{claim}\label{claim3.5}
    $L(H)$ is generated by all the $hL(\lbrace 1,s\rbrace)$, $h\in H$, $s\in S_{H}$.
\end{claim}

\begin{cproof}
Since $L(H)$ is generated by the subgroups $L(R)$, for finite subsets $R\subset H$, it is enough to show that every such $L(R)$ is a subgroup of the subgroup generated by the $hL(\lbrace1_{H},s\rbrace)$, for $h\in H$ and $s\in S_{H}$. We prove this fact by induction over the length $|R|_H$ of $R$. If $|R|_{H}=1$, then $R$ is included in $\lbrace 1_{H},s\rbrace$ for some $s\in S_{H}$, so $L(R)$ is a subgroup of $L(\lbrace 1,s \rbrace)$. Suppose that $|R|_H\ge 2$. By the decreasing length property, there exist finite subsets $R_{1},\dots,R_{k}$ of $H$ such that $L(R)$ is a subgroup of $\langle hL(R_{i}) : h\in H, 1\le i\le k\rangle$ and each $R_{i}$ has length less than the one of $R$, so we conclude by induction.
\end{cproof}

Now, Claim~\ref{claim3.5} ensures that any $(\sigma,1_{H})\in\halo H$ can be decomposed as a product of elements whose first coordinates lie in $\bigcup_{h\in H,\;s\in S_{H}}{hL(\lbrace1,s\rbrace)}$. Writing an element $h\in H$ as $h=s_{1}\dots s_{n}$ with $s_{1},\dots,s_{n}\in S_{H}$, we get
\begin{equation*}
    (hL(\lbrace 1,s\rbrace), 1_{H})=(1_{L(H)},s_{1})\dots (1_{L(H)},s_{n})(L(\lbrace 1,s\rbrace),1_{H})(1_{L(H)},s_{n}^{-1})\dots (1_{L(H)},s_{1}^{-1})
\end{equation*}
where $(L(R),1_{H})$ is a shorthand for the set $\lbrace (\sigma, 1_{H})\in\halo H : \sigma\in L(R)\rbrace$. This shows that $(\sigma,1_{H})$ belongs to the subgroup generated by $(1_{L(H)},S_H)\cup\bigcup_{s\in S_{H}}(L(\lbrace 1_{H},s\rbrace),1_{H})$, as was to be shown.
\end{proof}

Let us point out a simple criterion to check for a halo product, and that guarantees the decreasing length property, and thus the natural generation property.

\begin{definition}
Let $\halo H$ be a halo product over a finitely generated group $H$. We say that $\halo H$~\textit{has the gluing property} if for any subsets $R,S\subset H$ such that $R\cap S\neq\emptyset$, we have 
\begin{equation*}
    L(R\cup S)=\langle L(R), L(S)\rangle.
\end{equation*}
\end{definition}

\begin{example}\ 
\begin{itemize}
    \item Wreath products $F\wr H$, with a group $F$, have the gluing property.
    \item Lampshufflers also satisfy the gluing property. Indeed, given non disjoint subsets $R,S\subset H$, it suffices to prove that every transposition $\tau_{x,y}$ supported in $R\cup S$ lies in $\langle L(R),L(S)\rangle$. If $x$ and $y$ both lie in $R$ (or in $S$), it is obvious. Otherwise, let us assume $x\in R$ and $y\in S$, and let us consider $z\in R\cap S$. Then the conjugation by $\tau_{x,z}$ (which lies in $L(R)$) maps $\tau_{x,y}$ to $\tau_{z,y}$ (which lies in $L(S)$), which proves the claim.
    \item Lampcloners also have the gluing property, with a proof very similar to the case of lampshufflers: we use the fact that blocks are generated by transvections (playing the same role as the transposition $\tau_{x,y}$) and we conjugate by linear automorphisms acting as transpositions on the canonical basis given by $H$ (as $\tau_{x,z}$ for the lampshuffler).
    \item However, lampupcloners do not have the gluing property. Here is a counter-example. Consider the group $\Z$ with its usual order so that we can see elements of blocks over finite subsets as upper triangular matrices with diagonal entries equal to $1$. Let us consider $R=\lbrace 1,3\rbrace$ and $S=\lbrace 2,3\rbrace$. Then we have
\begin{equation*}
    L(\lbrace 1,2,3\rbrace)=\left\{\begin{pmatrix}
            1&a&b\\
            0&1&c\\
            0&0&1
    \end{pmatrix}\colon a,b,c\in\field\right\}
\end{equation*}
and 
\begin{equation*}
    L(\lbrace1,2\rbrace)=\left\{\begin{pmatrix}
            1&a&0\\
            0&1&0\\
            0&0&1
\end{pmatrix}\colon a\in\field\right\},\
    L(\lbrace1,3\rbrace)=\left\{\begin{pmatrix}
            1&0&b\\
            0&1&0\\
            0&0&1
        \end{pmatrix}\colon b\in\field\right\},
\end{equation*}
but $L(\lbrace 1,2\rbrace)$ and $L(\lbrace 1,3\rbrace)$ generate the group
\begin{equation*}
        \left\{\begin{pmatrix}
            1&a&b\\
            0&1&0\\
            0&0&1
    \end{pmatrix}\colon a,b\in\field\right\}
\end{equation*}
which is a proper subgroup of $L(\lbrace 1,2,3\rbrace)$.
\end{itemize}
\end{example}

\begin{proposition}\label{prop:finitegensetforhalo}
Let $H$ be a finitely generated group and let $S_{H}$ be a finite generating set of $H$. Let $\halo H$ be a halo product over $H$. If $\halo H$ has the gluing property, then it is naturally generated.
\end{proposition}

Note that, in the case of lampshufflers (more generally lampjugglers) and lampcloners, we recover the finite generating sets we exhibited in Section~\ref{sec:defHalo}.

\begin{proof}
We in fact check that the decreasing length property is satisfied and we conclude using Proposition~\ref{prop:DefEquivalentFG}. Let $R$ be a subset of $H$, of length $\geq 2$. The case $|R|=1$ is immediate, since we can write $L(R)=rL(\{1_H\})$, where $R=\{r\}$. Assume that $R$ has cardinality $2$. Let us write $R=\lbrace h_{1},h_{2}\rbrace$. Up to a translation, we can assume that $h_{2}=1_{H}$ and that $h_{1}\neq 1_{H}$. Then we write $h_{1}=s_{1}\ldots s_{n}$ with $s_{1},\dots,s_{n}\in S_{H}$ and we notice that 
\begin{equation*}
    L(\lbrace1_{H},h_{1}\rbrace)\leqslant L(\{1_{H},s_{1},s_{1}s_{2},\ldots,s_{1}\dots s_{n}\}).
\end{equation*}
Applying successively the gluing property, the latter is generated by 
\begin{equation*}
    L(\lbrace 1_{H},s_{1}\rbrace),\;L(\lbrace s_{1},s_{1}s_{2}\rbrace), \dots,\;L(\lbrace s_{1}\dots s_{n-1},s_{1}\dots s_{n}\rbrace).
\end{equation*}
As, for any $1\le i\le n$, we have $L(\lbrace s_{1}\dots s_{i-1},s_{1}\dots s_{i}\rbrace)=s_{1}\dots s_{i-1}L(\lbrace 1_{H},s_{i}\rbrace)$, it suffices to set $R_i=\{1_H,s_i\}$ (of length $1$) to get the desired property with respect to $R$.\par
If $R$ has cardinality greater than or equal to $3$, pick any two points $h,h'\in R$. Let us set $R_{1}=\lbrace h,h'\rbrace$ and $R_{2}=R\setminus\lbrace h'\rbrace$. Their union equals $R$ and they have non-empty intersection, so the gluing property implies that $L(R)$ is generated by $L(R_{1})$ and $L(R_{2})$. Furthermore, $|R_{1}|_{H}$ and $|R_{2}|_{H}$ are smaller than $|R|_{H}$ since $R_{1}$ and $R_{2}$ are proper subsets of $R$, so we are done.
\end{proof}

We can prove that some lampupcloners are naturally generated, even though they do not satisfy the gluing property.

\begin{proposition}\label{prop:UpclonerFG}
Let $\field$ be a field, let $d\geq 1$ be an integer, and endow $\Z^{d}$ with the lexicographic order. Then the lampupcloner $\upcloner{\Z^d}$ is naturally generated.
\end{proposition}

\begin{proof}
Let us first mention that, given a group $H$, transvections in $\mathrm{FGL}(H)$ behave well with respect to commutators, in the sense that:
\begin{equation}\label{eq:Commutators}
        \forall f,r,s\in H,\ \forall\lambda,\mu\in\field,\ \tau_{r,f}(-\lambda)\tau_{f,s}(-\mu)\tau_{r,f}(\lambda)\tau_{f,s}(\mu)=\tau_{r,s}(\lambda).
\end{equation}
Our goal is to use this identity when $r<f<s$, for the transvections to be in $\text{FU}(H)$.

\smallskip 

Let us now prove that $\upcloner{\Z^d}$ has the decreasing length property, where $\Z^d$ is endowed with the lexicographic order and with its canonical basis $\lbrace e_{1},\dots,e_{d}\rbrace$ as a finite generating set. Let $R$ be a finite subset of $\Z^d$, of length $\ge 2$, and let us enumerate its elements in an ordered way: $R=\lbrace r_{1}<\dots <r_{n}\rbrace$, with $n=|R|$. The case $|R|=1$ is immediate since we can write $\text{FU}(R)=r_{1}+\text{FU}(\lbrace 0\rbrace)$. If $|R|\ge  3$, then we set $R_{1}=\lbrace r_{1},r_{2}\rbrace$ and $R_{2}=\lbrace r_{2},r_{3},\dots,r_{n}\rbrace$, whose lengths are less than the one of $R$, and we apply~\eqref{eq:Commutators} to $r=r_{1}$, $f=r_{2}$ and $r=r_{i}$ for each $i\in\lbrace 3,\ldots,n\rbrace$, to get that $\text{FU}(R)$ is generated by $\text{FU}(R_{1})$ and $\text{FU}(R_{2})$. Let us finally assume that $R$ has cardinality $2$. Since $r_{1}<r_{2}$, we can write
\begin{equation*}
        r_2-r_1=(0,\ldots,0,k_i,k_{i+1},\ldots,k_d),
\end{equation*}
with $k_{i}\ge 1$. We know that $|r_{2}-r_{1}-e_{i}|_{\Z^d}<|r_{2}-r_{1}|_{\Z^d}\le  |r_{2}|_{\Z^d}+|r_{1}|_{\Z^d}$, namely $|r_{2}-r_{1}-e_{i}|_{\Z^d}<|R|_{\Z^d}$. There are several cases to consider.
\begin{enumerate}
    \item If $k_{i}\ge 2$, then $r_{2}-r_{1}-e_{i}>0$. We thus have $r_{1}<r_{1}+e_{i}<r_{2}$, so applying~\eqref{eq:Commutators} to $r=r_{1}$, $s=r_{2}$ and $f=r_{1}+e_{i}$, we get that
    \begin{equation*}
            \text{FU}(R)\leqslant \text{FU}(\lbrace r_{1},r_{1}+e_{i},r_2\rbrace)=\langle r_{1}+\text{FU}(\lbrace 0,e_{i}\rbrace), r_{1}+e_{i}+\text{FU}(\lbrace 0,r_{2}-r_{1}-e_{i}\rbrace)\rangle.
    \end{equation*}
    We are done since we have $|\lbrace 0,e_{i}\rbrace|_{\Z^d}=1<|R|_{\Z^d}$ and $|\lbrace 0,r_{2}-r_{1}-e_{i}\rbrace|_{\Z^d}<|R|_{\Z^d}$.
    \item If $r_{2}-r_{1}=e_{i}$, then $\text{FU}(R)=r_{1}+\text{FU}(\lbrace 0,e_{i}\rbrace)$ and we are done.
    \item We finally assume $k_{i}=1$ and $r_{2}-r_{1}\neq e_{i}$, so that we can define $i_{0}\defeq\min{\lbrace j\geq i+1 : k_j\neq 0\rbrace}$. Let us set $h\defeq e_{i}+(k_{i_0}-1)e_{i_0}$. We have
    \begin{equation*}
            (r_{2}-r_{1}-h)_{j}=\left\{\begin{array}{ll}
            (r_{2}-r_{1})_{j} &\text{ if }j\in\lbrace 1,\dots,d\rbrace\setminus\lbrace i,i_{0}\rbrace \\
                0 & \text{ if }j=i\\
                1 & \text{ if }j=i_{0}
            \end{array}\right.,
    \end{equation*}
    so that $r_{2}-r_{1}-h>0$ and $|r_{2}-r_{1}-h|_{\Z^d}<|r_{2}-r_{1}|_{\Z^d}\le |R|_{\Z^d}$. We also have $h>0$, so we get
    \begin{equation*}
        \text{FU}(R)\leqslant L(\lbrace r_{1},r_{1}+h,r_{2}\rbrace)\leqslant\langle r_{1}+\text{FU}(\lbrace 0,h\rbrace), r_{1}+h_{1}+\text{FU}(\lbrace 0,r_{2}-r_{1}-h\rbrace)\rangle
    \end{equation*} with the same techniques, and we similarly claim that we are done.
\end{enumerate}
So $\upcloner{\Z^d}$ satisfies the decreasing length property.
\end{proof}

\subsubsection{Growth of lamps} 

We conclude this section by recalling an important definition from~\cite{GT24a}, that of the~\textit{lamp growth sequence} associated to a halo product. In order to define this sequence, we need an additional mild assumption on our halo products. 

\begin{definition}
Let $H$ be a group. We say that a halo product $\halo$ over $H$ is~\textit{consistent} if it has finite blocks and if, for any finite subset $S\subset H$, the cardinality of $L(S)$ only depends on the cardinality of $S$.
\end{definition}

In practice, consistent halo products encompass all classes we are interested in, among which lamplighters, lampjugglers, lampdesigners, lampcloners and lampupcloners.

\begin{definition}
Let $\halo H$ be a consistent halo product over a group $H$.
The~\textit{lamp growth sequence} of $\halo H$ is the function $\Lambda_{\halo H}\colon \N\longrightarrow\N$ defined by 
\begin{equation*}
    \Lambda_{\halo H}\colon n\longmapsto |L(S)|, \;\text{where}\; |S|=n.
\end{equation*}
\end{definition}

This sequence is well-defined since $\halo$ is consistent. It has been computed in~\cite[Facts~7.12-7.16]{GT24a} for many halo products, such as lamplighters, lampshufflers, lampdesigners, lampcloners and $2-$nilpotent wreath products. For instance, for any group $H$ and finite group $F$, one has $\Lambda_{F\wr H}(n)=|F|^{n}$, $\Lambda_{\juggler{r}{H}}(n)=(rn)!$ and $\Lambda_{\designer{H}}(n)=|F|^{n}n!$. Moreover, given a finite field $\field$, we have
\begin{equation*}
    \Lambda_{\cloner{H}}(n)=\prod_{i=0}^{n-1}{(|\field|^n-|\field|^i)}
\end{equation*}
and, if $H$ is a totally ordered group,
\begin{equation*}
    \Lambda_{\upcloner{H}}(n)=\prod_{i=1}^{n-1}{|\field|^i}=|\field|^{\frac{n(n-1)}{2}}.
\end{equation*}

\smallskip

In fact, it turns out that the asymptotic behaviour of this sequence is invariant under a special class of quasi-isometries (and more generally coarse embeddings), referred to in~\cite[Section~6]{GT24a} as~\textit{aptolic quasi-isometries}. As proved in~\cite[Corollary~6.11]{GT24a}, under additional assumptions, any quasi-isometry between two halo products is aptolic (up to finite distance). Thus, for these halo products, the asymptotic behaviour of the lamp growth sequence is an invariant of quasi-isometry. 

\section{Estimates of the F\o lner functions of halo products}\label{sec:computationsFolnerFunction}

\subsection{A general upper bound on the F\o lner functions: finding almost invariant functions}\label{sec:UpperBoundFolner} 

In this section, we provide an upper bound on the $\ell^p$-F\o lner function of many halo products, such as lampshufflers, lampjugglers, lampdesigners, lampcloners or lampupcloners over $\Z^d$. The $\ell^p$-F\o lner function of a finitely generated group $G$ being an infimum over finitely supported functions $G\longrightarrow\R$, the strategy is to exhibit "good" such functions, namely almost invariant functions. This constitutes the first step towards proving Theorem~\ref{th:recallThm intro}.

\smallskip

Observe that, if $H$ is an amenable group, then $\halo H$ is amenable if and only if $L(H)$ is amenable, and the latter is often true regardless of $H$. For instance, if blocks are finite, then $L(H)$ is locally finite and thus amenable. This is the case when $\halo H$ is a lamplighter $F\wr H$ (i.e. $F$ is finite), a lampshuffler $\shuf{H}$, a lampcloner $\cloner{H}$ over a finite field $\field$, or a lampupcloner $\upcloner{H}$.

\smallskip

Therefore, we know that almost invariant functions exist when our halo product has finite blocks. Here our goal is, in particular, to construct a suitable sequence of almost invariant functions for our halo product from such a sequence for the base group $H$. The estimates we can derive enable us to prove the following.

\begin{proposition}\label{prop:upperboundonFol}
Let $H$ be a finitely generated amenable group, and let $\halo H$ be a halo product over $H$. Suppose that $\halo H$ is naturally generated and has consistent blocks. Then, for any $p\geq 1$, there exists a constant $C>0$ such that  
\begin{equation*}
    \folp{\halo H}(x)\preccurlyeq\folp{H}(x)\cdot \Lambda_{\halo H}(C\cdot\folp{H}(x)),
\end{equation*}
where $\Lambda_{\halo H}$ is the lamp growth sequence of $\halo H$. 
\end{proposition}

The proof is inspired by the one of~\cite[Theorem~6.8]{SCZ21}.

\begin{proof}
As usual, denote $S_{H}$ a finite generating set for $H$. By assumption, the subset
\begin{equation*}
    S_{\halo H}=\left\lbrace (1_{L(H)},s) : s\in S_H\right\rbrace \cup\bigcup_{s\in S_{H}}{\left\lbrace (\sigma_{s},1_{H}) : \sigma\in L(\lbrace 1_{H},s\rbrace)\right\rbrace}
\end{equation*}
generates $\halo H$. Let $(f_{n})_{n\ge 0}$ be a sequence of functions $H\longrightarrow \R$ that realises $\folp{H}$, i.e. $\folp{H}(n)=|\supp f_{n}|$ and $\frac{\|\nabla_{S_H}f_{n}\|_{p}}{\|f_{n}\|_{p}} \le \frac{1}{n}$ for any $n\ge 0$. Given $n\ge 0$, set 
\begin{equation*}
    U_{n}\defeq\supp{f_{n}},\; V_{n}\defeq\bigcup_{s\in S_{H}}{U_{n}s}
\end{equation*}
and
\begin{equation*}
    \begin{array}{llcl}
    g_{n}\colon &\halo H &\longrightarrow &\R\\
    &(\sigma,h)&\longmapsto &f_{n}(h)\mathds{1}_{\sigma\in L(V_n)}
    \end{array}.
\end{equation*}
Let $(\sigma,h)\in\halo H$, $s\in S_{H}$ and $\sigma_{s}\in L(\lbrace 1,s\rbrace)$. The composition law of $\halo H$ directly implies that $(\sigma,h)(1_{L(H)},s)=(\sigma,hs)$ and $(\sigma,h)(\sigma_{s},1_{H})=(\sigma(h\cdot\sigma_{s}),h)$. This implies that 
\begin{equation*}
    g\left ((\sigma,h)(1_{L(H)},s)\right )-g(\sigma,h)=(f(hs)-f(h))\mathds{1}_{\sigma\in L(V_n)}
\end{equation*}
as well as
\begin{equation*}
    g\left((\sigma,h)(\sigma_{s},1_{H})\right)-g(\sigma,h)=0
\end{equation*}
using that $\sigma\in L(V)$ if and only if $ \sigma(h\cdot\sigma_s)\in L(V)$ when $h\in U_{n}$. We thus have
\begin{align*}
    \|\nabla_{S_{\halo H}}g_n\|_{p}^{p}&=\sum_{(\sigma,h)\in\halo H}{\sum_{s\in S_{H}}{|g_{n}\left((\sigma,h)(1_{L(H)},s))\right)-g_{n}\left((\sigma,h)\right)|^p}}\\
    &=\sum_{(\sigma,h)\in\halo H}{\sum_{s\in S_{H}}{|(f(hs)-f(h))\mathds{1}_{\sigma\in L(V_{n})}|^p}}\\
    &=|L(V_{n})|\cdot\|\nabla_{S_H}f_{n}\|_{p}^{p}.
\end{align*}
We also have
\begin{equation*}
    \|g_{n}\|_{p}^{p}=|L(V_{n})|\cdot\|f_{n}\|_{p}^{p}.
\end{equation*}
We finally get
\begin{equation*}
\frac{\|\nabla_{S_{\halo H}}g_{n}\|_{p}}{\|g_{n}\|_{p}}=\frac{\|\nabla_{S_{H}}f_{n}\|_{p}}{\|f_{n}\|_{p}}\le\frac{1}{n}
\end{equation*}
so, by the definition of the $\ell^p$-F\o lner function, it follows that
\begin{equation*}
    \folp{\halo H}(n)\le |\supp g_{n}|=|L(V_{n})|\cdot |U_{n}|=|U_{n}|\cdot \Lambda_{\halo H}(|V_{n}|)\le \folp{H}(n)\cdot\Lambda_{\halo H}\left(|S_{H}|\cdot \folp{H}(n)\right).
\end{equation*}
This concludes the proof.
\end{proof}

\subsection{A general lower bound on the F\o lner functions: finding lamplighter subgraphs}\label{sec:secondstrategy}

The goal of this section is to find a lower bound of the $\ell^p$-F\o lner function of a halo product. In Section~\ref{sec:computationsIsoProf}, we deduce, for specific cases, an upper bound of the $\ell^p$-isoperimetric profile which will often be optimal.

\smallskip 

Here is then the general statement. 

\begin{theorem}\label{th:lowerboundonFol}
Let $H$ be a finitely generated amenable group and let $S_{H}$ be a finite generating set. Let $\halo H$ be a naturally generated and large-scale commutative halo product having finitely generated blocks. Then,
for any $p\geq 1$ and any $s_{0}\in S_{H}$, there exists a constant $C>0$ such that
\begin{equation*}
    \folp{\halo H}(x) \succcurlyeq \left (\fol{L(\{1_H,s_0\})}(x)\right )^{C\fol{H}(x)}.
\end{equation*}
\end{theorem}

\smallskip

In the particular case of a halo product with finite blocks, we thus get
\begin{equation*}
    \folp{\halo H}(x) \succcurlyeq K^{\fol{H}(x)}
\end{equation*}
for some positive constant $K>0$ and any $p\ge 1$.

\smallskip

The following lemma will allow us to reduce to the case $p=1$. This is a well-known result on isoperimetric profiles, mentioned in~\cite{Cou00}, which states that for every finitely generated group $G$, $\profp{G}$ is monotonuous in the variable $p\ge 1$ for the order given by $\preccurlyeq$. Here we state it in terms of F\o lner functions and we provide a proof for the sake of completeness. Recall that it is conjectured that the asymptotic behaviour of $\profp{G}$ does not depend on $p$.

\begin{lemma}[Folklore]\label{lem:MonotonuousProfile}
Let $G$ be a finitely generated group and let $p,q$ be real numbers such that $p>q\geq 1$. Then we have $\folp{G}(x)\succcurlyeq\textup{F\o l}_{q,G}(x)$.
\end{lemma}

For Theorem~\ref{th:lowerboundonFol}, we will apply this lemma to $q=1$.

\begin{proof}
Let $f\colon G\longrightarrow \R$ be a finitely supported function, and consider the function $h\defeq |f|^{v}$ where $v\defeq \frac{p}{q}>1$. Using the inequality $|a^v-b^v|\leq v\max{(a,b)}^{v-1}|a-b|$ that holds for every positive real numbers $a,b\ge 0$, we get
\begin{align*}
        \|\nabla_{S_{G}} h\|_{q}^{q}&=\sum_{g\in G,\; s\in S_{G}}{|h(g)-h(gs)|^q}\\
        &\leq v^{q}\sum_{g\in G,\; s\in S_{G}}|f(g)|^{q(v-1)}\big||f(g)|-|f(gs)|\big|^{q}+v^q\sum_{g\in G,\; s\in S_{G}}|f(gs)|^{q(v-1)}\big| |f(g)|-|f(gs)|\big|^{q}\\
        &\leq v^{q}\sum_{g\in G,\; s\in S_{G}}|f(g)|^{q(v-1)}|f(g)-f(gs)|^{q}+v^{q}\sum_{g\in G,\; s\in S_{G}}|f(gs)|^{q(v-1)}|f(g)-f(gs)|^{q}.
\end{align*}
Setting $P=\frac{p}{q}$ and $Q=\frac{p}{p-q}$, we have $\frac{1}{P}+\frac{1}{Q}=1$, and Hölder's inequality provides
\begin{align*}
    \sum_{g\in G,\; s\in S_{G}}{|f(g)|^{q(v-1)}|f(g)-f(gs)|^q} &\le \left(\sum_{g\in G,\; s\in S_{G}}{|f(g)|^{Qq(v-1)}}\right)^{\frac{1}{Q}}\left(\sum_{g\in G,\; s\in S_{G}}{|f(g)-f(gs)|^{Pq}}\right)^{\frac{1}{P}}\\
    &=\left(\sum_{g\in G,\; s\in S_{G}}{|f(g)|^{p}}\right)^{\frac{p-q}{p}}\left(\sum_{g\in G,\; s\in S_{G}}{|f(g)-f(gs)|^{p}}\right)^{\frac{q}{p}}\\
    &=|S_{G}|^{\frac{p-q}{p}}\cdot \|f\|_{p}^{p-q}\cdot \|\nabla_{S_G}f\|_{p}^{q}
\end{align*}
and similarly for $\sum_{g\in G,\; s\in S_{G}}{|f(gs)|^{q(v-1)}|f(g)-f(gs)|^q}$, so that we get
\begin{equation*}
        \|\nabla_{S_{G}} h\|_{q} \le 2^{\frac{1}{q}}\cdot |S_G|^{\frac{p-q}{pq}}\cdot v\cdot \|f\|_{p}^{\frac{p-q}{q}}\cdot \|\nabla_{S_G}f\|_{p}
\end{equation*}
which in turn implies
\begin{equation*}
        \frac{\|\nabla_{S_{G}} h\|_q}{\|h\|_q} \le 2^{\frac{1}{q}}\cdot |S_{G}|^{\frac{p-q}{pq}}\cdot v \cdot \frac{\|\nabla_{S_{G}} f\|_p}{\|f\|_p}
\end{equation*}
since $\|h\|_{q}^{q}=\|f\|_{p}^{p}$. This inequality holds for every finitely supported function $f\colon G\longrightarrow \R$, namely we proved that for every such $f$, this inequality holds for some $h\colon G\longrightarrow\R$ having the same support, so the statement follows directly from the definition of F\o lner functions.
\end{proof}

We now move on to the proof of Theorem~\ref{th:lowerboundonFol}. At first reading, the reader may look at Appendix~\ref{appendixA}, where the strategy is to find "good" lamplighters as subgroups of a lampshuffler $\shuf{H}$, namely a lamplighter group based on a finitely generated subgroup $K$ of $H$ having the same isoperimetric profile as $H$. This is achieved with some algebraic assumptions on the finitely generated group $H$, covering a large class of groups. We finally conclude using the result analogous to Theorem~\ref{th:profilemonotonuousSubgroup} for the F\o lner function. Moreover, this first strategy provides an interesting framework since the algebraic assumptions on the base group $H$ are stable in many cases when taking iterations of lampshufflers; see Remark~\ref{rem:Preservation} and Proposition~\ref{prop:A7}.

\smallskip

In this section, we focus on a less restrictive substructure than subgroups, namely subgraphs. The strategy is to find some subgraph $X_0$ of $H$, quasi-isometric to it, playing the role of a subgroup $K$ as described in the above first strategy, and a lamplighter graph on $X_0$ as a subgraph of $\halo H$, in such a manner that we can prove the monotonocity of the F\o lner function in this context, as in Theorem~\ref{th:profilemonotonuousSubgroup}. We conclude thanks to the lower bounds for the F\o lner functions of lamplighter graphs obtained in~\cite{Ers06}.

\paragraph{Lamplighter graphs.} Let $A$ and $B$ be two graphs, with a base vertex $b_0$ in $B$. Given a map $f\colon A\longrightarrow B$, we define its support by $\supp{f}\defeq\lbrace a\in A : f(a)\not=b_{0}\rbrace$. The lamplighter graph of $B$ and $A$, denoted by $B\wr A$, is the graph
\begin{itemize}
    \item whose vertices are pairs $(f,a)$, where $a$ is a vertex of $A$ and $f\colon A\longrightarrow B$ has finite support;
    \item whose edges connect $(f,a)$ and $(f',a')$ if either $a=a'$, $f(a)\sim_{B}f'(a)$ and $f(v)=f'(v)$ for every $v\in A\setminus\lbrace a\rbrace$, or if $f=f'$ and $a\sim_{A}a'$.
\end{itemize}

In the case where the graphs $A$ and $B$ are Cayley graphs of finitely generated groups $G$ and $H$ respectively, we recover a Cayley graph of the wreath product $H\wr G$.

\smallskip

Let us now prove Theorem~\ref{th:lowerboundonFol} within this framework, using the following lower bound proved by Erschler~\cite[Theorem~4.3]{Ers06}: there exists $C>0$ such that
\begin{equation*}
    \fol{B\wr A}(x)\succcurlyeq\left (\fol{B}(x)\right)^{C\fol{A}(x)}.
\end{equation*}

\begin{proof}[Proof of Theorem~\ref{th:lowerboundonFol}]
By Lemma~\ref{lem:MonotonuousProfile}, we have $\folp{\halo H}(x)\succcurlyeq\fol{\halo H}(x)$, so it is enough to prove the theorem for $p=1$.

\smallskip

Given a connected graph $Y$, we denote by $d_{Y}(\cdot,\cdot)$ its path metric. When considering a finitely generated group $H=\langle S\rangle$, we write $d_{H,S}(\cdot,\cdot)$ for the path metric on its Cayley graph $\text{Cay}(H,S)$ (identified with $H$ itself), to specify the choice of a finite generating subset $S$.

\smallskip

Now, let us fix a finite generating set $S$ of $H$, a constant $D\geq 0$ of large-scale commutativity for $\halo H$, and let us consider $S_{2D+5}\defeq\lbrace s_{1}s_2\ldots s_{2D+5} : s_{i}\in S\cup\lbrace 1_{H}\rbrace\rbrace$. Note that, for every $x,y\in H$, we have the equivalence 
\begin{equation*}
    d_{H,S}(x,y)\le 2D+5\Longleftrightarrow d_{H,S_{2D+5}}(x,y)\le 1.
\end{equation*}
Let $X_{0}$ be a maximal $(D+2)$-separated subset of $H$, for the metric $d_{H,S}$, and let us endow $X_{0}$ with the graph structure induced by $d_{H,S_{2D+5}}$, namely $x,y\in X_{0}$ are adjacent if and only if $d_{H,S_{2D+5}}(x,y)=1$. It is straightforward to see that $(X_{0},d_{X_{0},S_{2D+5}})$ is a subgraph of $(H,d_{H,S_{2D+5}})$. In addition, we also have the following.
    
\begin{claim}\label{cl:qiToH}
The graphs $(X_{0},d_{X_{0},S_{2D+5}})$ and $(H,d_{H,S_{2D+5}})$ are quasi-isometric.
\end{claim}

\begin{cproof}
Let us prove that the natural inclusion $X_{0}\hookrightarrow H$ is a quasi-isometry. First of all, by maximality, $X_{0}$ is $(D+2)$-dense in $(H,d_{S})$, and this directly implies $d_{H,S_{2D+5}}(h,X_{0})\le 1$ for every $h\in H$.
        
\smallskip
        
Let $x,y\in X_0$. It is straightforward to show that $d_{X_0,S_{2D+5}}(x,y)\ge d_{H,S_{2D+5}}(x,y)$. The other way around, let $n\defeq d_{H,S_{2D+5}}(x,y)$. By definition, there exist 
points 
\begin{equation*}
    x_0=x,x_1,\ldots,x_{n-1},x_n=y\in H
\end{equation*}
such that $d_{H,S_{2D+5}}(x_{i},x_{i+1})=1$ for every $i\in\lbrace 0,1,\ldots,n-1\rbrace$. Given such an index $i$, the definition of $S_{2D+5}$ implies that there exist points
\begin{equation*}
x_{i,0}=x_{i},x_{i,1},\ldots,x_{i,2D+4},x_{i,2D+5}=x_{i+1}\in H
\end{equation*}
such that $d_{H,S}(x_{i,j},x_{i,j+1})\le 1$ for every $j\in\lbrace 0,1,\ldots, 2D+4\rbrace$. Since we have $x_{i,2D+5}=x_{i+1,0}$ for every $i\in\lbrace 0,1,\ldots,n-1\rbrace$, we have found a path of length $\le (2D+5)n$ in $(H,d_{S})$ that connects $x$ to $y$. Approximating every vertex of this path by an element of $X_{0}$ within $d_{H,S}$-distance less than $D+2$ ($x$ and $y$ being approximated by themselves), we get a sequence 
\begin{equation*}
   w_{0}=x,w_{1},\ldots,w_{(2D+5)n-1},w_{(2D+5)n}=y 
\end{equation*}
of elements in $X_{0}$ satisfying $d_{H,S}(w_{i},w_{i+1})\le 2(D+2)+1=2D+5$, whence $d_{H,S_{2D+5}}(w_{i},w_{i+1})\le 1$. This way, we get a path from $x$ to $y$, of length $\le (2D+5)n$, in $(X_{0},d_{S_{2D+5}})$. Thus 
\begin{equation*}
 d_{X_{0},S_{2D+5}}(x,y)\le (2D+5)\cdot d_{H,S_{2D+5}}(x,y)   
\end{equation*}
and the proof of the claim is complete.
\end{cproof}
    
Let us fix some distinguished generator $s_0\in S\setminus\lbrace 1_{H}\rbrace$. By $(D+2)$-separation, for every $x\in X_{0}$, $xs_{0}$ does not lie in $X_{0}$, and large-scale commutativity thus implies that the groups $L(\lbrace x,xs_{0}\rbrace)$, for $x\in X_{0}$, commute. Let us now introduce the subgroup $\mathcal{T}$ of $L(H)$ defined by
\begin{equation*}
    \mathcal{T}\defeq\left\lbrace\prod_{x\in I}{\sigma_x}\ :  I \subset X_{0}\;\text{is finite}, \sigma_{x}\in L(\lbrace x,xs_{0}\rbrace)\right\rbrace=\bigoplus_{x\in X_{0}}{L(\lbrace x,xs_{0}\rbrace)}=\bigoplus_{x\in X_{0}}{\alpha(x)L(\lbrace 1_{H},s_{0}\rbrace)}.
\end{equation*}
Since $\halo H$ has finitely generated blocks, we can fix a finite generating subset $S(s_{0})$ of $L(\lbrace 1_{H},s_{0}\rbrace)$.
Let us now consider the set $Y_{\star}\defeq\mathcal{T}\times X_{0}$ equipped with a graph structure where two vertices $(\rho,x)$ and $(\rho',x')$ of $Y_\star$ are adjacent if 
\begin{itemize}
        \item either $x=x'$ and $\rho^{-1}\rho'=\alpha(x)(\sigma)$ for some $\sigma\in S(s_0)$;
        \item or $\rho=\rho'$ and $d_{X_0,S_{2D+5}}(x,x')= 1$,
\end{itemize}
which can be reformulated as
\begin{itemize}
        \item either $(\rho,x)(\sigma,1_H)=(\rho',x')$, for some $\sigma\in S(s_0)$;
        \item or $(\rho,x)(\mathrm{id}_{H},h)=(\rho',x')$, where $h$ lies in $S_{2D+5}$.
\end{itemize}
Note that the graph $Y_{\star}$ is isomorphic to the lamplighter graph $L(\lbrace 1_{H},s_{0}\rbrace)\wr X_0$. Thus, we endow $\halo H$ with the finite generating set $S_{\halo H}$ given by
\begin{equation*}
    S_{\halo H}\defeq\lbrace (\sigma,1_{H}) : \sigma\in S(s_{0})\rbrace \cup\lbrace (1_{L(H)},h) : h\in S_{2D+5}\rbrace.
\end{equation*}
Now, let us consider the partition of $L(H)$ in $\mathcal{T}$-cosets:
\begin{equation*}
        L(H)=\bigsqcup_{c\in C}{\kappa_{c}\mathcal{T}}
\end{equation*}
with $\kappa_{c_{0}}=1_{L(H)}$ for the index $c_{0}\in C$ of the coset $\mathcal{T}$. For every $c\in C$, let us consider the subset
\begin{equation*}
    Y_{c}\defeq (\kappa_{c}\mathcal{T})\times H,
\end{equation*}
equipped with a graph structure where two vertices $(\kappa_{c}\rho,x)$ and $(\kappa_{c}\rho',x')$ are adjacent if 
\begin{itemize}
\item either $x=x'$, $x$ lies in $X_{0}$ and $\rho^{-1}\rho'=\alpha(x)(\sigma)$ for some $\sigma\in S(s_{0})$, namely $(\kappa_{c}\rho,x)(\sigma,1_{H})=(\kappa_{c}\rho',x')$;
\item or $\rho=\rho'$ and $d_{H,S_{2D+1}}(x,x')= 1$, namely $(\kappa_{c}\rho,x)(1_{L(H)},h)=(\kappa_{c}\rho',x')$ where $h$ lies in $S_{2D+1}$.
\end{itemize}
By definition of $S_{\halo H}$, $(Y_{c})_{c\in C}$ is a family of subgraphs of $(\halo H,d_{S_{\halo H}})$ partitioning the set of its vertices. Moreover the graph $Y_{c}$ is the left translation by $(\kappa_{c},1_{H})$ of the graph $Y_{c_{0}}$. These observations imply
\begin{equation*}
    \fol{\halo H}(n)\succcurlyeq \fol{Y_{c_{0}}}(n),
\end{equation*}
as an immediate adaptation of~\cite[Lemma~4]{Ers03}. The next claim is the final step required for the proof. 

\begin{claim}\label{cl:piecesQI}
The graph $Y_{c_{0}}$ is quasi-isometric to $Y_{\star}$.
\end{claim}

\begin{cproof}
The proof relies on the same technique as in the proof of Claim~\ref{cl:qiToH}. We prove that $Y_{\star}\hookrightarrow Y_{c_{0}}$ is a quasi-isometry. The $1$-density of its image is straightforward, as well as the inequality 
\begin{equation*}
    d_{Y_\star}((\rho,x),(\rho',x'))\ge d_{Y_{c_{0}}}((\rho,x),(\rho',x'))
\end{equation*}
for every $(\rho,x),(\rho',x')\in Y_{\star}$.

\smallskip
        
The other way around, notice that edges of a path of length $n\defeq d_{Y_{c_{0}}}((\rho,x),(\rho',x'))$ in $Y_{c_{0}}$ consists in either modifying the permutation on the first coordinate, or moving the arrow pointing at some element of $H$ in the second coordinate. Thus, with the same ideas as in the proof of Claim~\ref{cl:qiToH}, it suffices to approximate elements in the second coordinate by elements of $X_{0}$, so that we get a new path in $Y_{\star}$ of length $\le (2D+5)n$. This concludes the proof.
\end{cproof}
    
Combining the above claims, we finally get that 
\begin{equation*}
    \fol{\halo H}(n) \succcurlyeq \fol{Y_{c_{0}}}(n) \simeq \fol{Y_{\star}}(n) \simeq \fol{L(\{1_H,s_0\})\wr X_{0}}(n)
\end{equation*}
and the latter dominates $\left (\fol{L(\lbrace 1_{H},s_{0}\rbrace)}(n)\right )^{C'\fol{X_0}(n)}$, for some constant $C'>0$, using~\cite[Theorem~4.3]{Ers06}. From Claim~\ref{cl:qiToH}, $\fol{X_0}$ is asymptotically equivalent to $\fol{H}$, and thus 
\begin{equation*}
   \fol{\halo H}(n) \succcurlyeq \left(\fol{L(\lbrace 1_{H},s_{0}\rbrace)}(n)\right)^{C\fol{H}(n)} 
\end{equation*}
for some constant $C>0$.
\end{proof}

\section{Estimates of isoperimetric profiles for some examples of halo products}~\label{sec:computationsIsoProf}

The goal of this section is to establish our estimates of $\ell^p$-isoperimetric profiles of many halo products and their iterated versions, applying our estimates on F\o lner functions. Here we use the fact that the $\ell^p$-F\o lner function and the $\ell^p$-isoperimetric profile are generalized inverses of each other, and even that we may assume without loss of generality that they are inverses of each other, using Remark~\ref{rem:InverseEachOther}.\par
In the applications, we will make use of the following condition.

\begin{definition}
We say that a non-decreasing map $h\colon\R_{+}\longrightarrow \R_{+}$ satisfies \textit{Assumption}~$(\star)$ if 
\begin{equation*}
    \forall C>0,\; h(Cx)=O(h(x)). 
\end{equation*}
\end{definition}

Assumption~$(\star)$ already appeared in the literature~\cite{Ers03,Corr24} in the case where $h=\profp{H}$ is the $\ell^p$-isoperimetric profile of a finitely generated group $H$, and it seems that $\profp{H}$ satisfies this assumption for many choices of groups $H$. In fact, to our knowledge, there is currently no known example of a finitely generated group whose $\ell^p$-isoperimetric profiles do not satisfy Assumption~$(\star)$.

\subsection{General estimates on the isoperimetric profiles}

Recall that we proved in the previous section the following.

\begin{theorem}[see Proposition~\ref{prop:upperboundonFol} and Theorem~\ref{th:lowerboundonFol}]\label{th:recallThm}
Let $p\ge 1$. Let $H$ be a finitely generated amenable group and let $S_{H}$ be a finite generating set. Let $\halo H$ be a naturally generated halo product over $H$.
\begin{enumerate}[label=(\roman*)]
    \item If $\halo H$ is large-scale commutative and has finitely generated blocks, then for any $s_{0}\in S_{H}$, there exists a constant $C>0$ such that
\begin{equation*}
    \left(\fol{L(\lbrace 1_{H},s_{0} \rbrace)}(x)\right)^{C\fol{H}(x)} \preccurlyeq \folp{\halo H}(x).
\end{equation*}
    \item If $\halo H$ has consistent blocks, then there exists a constant $C>0$ such that
\begin{equation*}
    \folp{\halo H}(x) \preccurlyeq \folp{H}(x)\cdot \Lambda_{\halo H}(C\cdot \folp{H}(x)).
\end{equation*}
\end{enumerate}
\end{theorem}

As an easy consequence, if $\halo H$ is large-scale commutative, naturally generated and has consistent blocks, then its F\o lner function satisfies
\begin{equation*}
    K^{\fol{H}(x)}\preccurlyeq\folp{\halo H}(x) \preccurlyeq \folp{H}(x)\cdot \Lambda_{\halo H}(C\cdot \folp{H}(x))
\end{equation*}
for some positive constants $C,K>0$. We now reformulate these inequalities in terms of isoperimetric profile.

\begin{corollary}\label{cor:boundonProf}
Let $H$ be a finitely generated amenable group. Let $\halo H$ be a naturally generated and large-scale commutative halo product having consistent blocks. Then, given $p\geq 1$, the $\ell^p$-isoperimetric profile of $\halo H$ satisfies
\begin{equation*}
    \frac{1}{C}\profp{H}\left (\frac{1}{C}\varphi^{-1}\left (x\right )\right )\leq\profp{\halo H}(x)\leq C\prof{H}(C\ln{x})
\end{equation*}
for some positive constant $C>0$, and with $\varphi(x)=x\cdot \Lambda_{\halo H}(x)$, where $\Lambda_{\halo H}$ is the lamp growth sequence of $\halo H$. 
\end{corollary}

\begin{proof}
The inequality $K^{\fol{H}(y)} \preccurlyeq \folp{\halo H}(y)$ provided by Theorem~\ref{th:lowerboundonFol} implies that there exists a constant $C_1>0$ such that
\begin{equation*}
\fol{H}(y)\leq C_1\ln{(\folp{\halo H}(C_1y))},
\end{equation*}
namely $y\leq \prof{H}(C_1\ln{(\folp{\halo H}(C_1y))})$. It suffices to take $y=\frac{\profp{\halo H}(x)}{C_1}$ to deduce a first inequality:
\begin{equation*}
    \profp{\halo H}(x)\leq C_1\prof{H}(C_1\ln{x}).
\end{equation*}
For the second inequality, we know from Proposition~\ref{prop:upperboundonFol} that there is a constant $C_2>0$ such that
\begin{equation*}
    \folp{\halo H}(x) \le \varphi(C_2\cdot\folp{H}(C_2x))
\end{equation*}
for all $x$ large enough. Putting $x=\frac{\profp{H}(y)}{C_2}$ in this inequality, and applying $\profp{\halo H}$ which is increasing, one gets
\begin{equation*}
    \frac{\profp{H}(y)}{C_2} \le \profp{\halo H}(\varphi(C_2y))
\end{equation*}
for all $y$ large enough. It remains to take $y=\frac{1}{C_2}\varphi^{-1}\left (x\right )$ and we get
\begin{equation*}
    \frac{1}{C_2}\profp{H}\left (\frac{1}{C_2}\varphi^{-1}\left (x\right )\right )\leq\profp{\halo H}(x).
\end{equation*}
The corollary finally holds with $C\defeq\max{(C_1,C_2)}$.
\end{proof}

From this corollary, we can deduce estimates of the isoperimetric profiles of iterated halo products $\halo^{\circ n} H$ defined by $\halo^{\circ 0} H=H$ and $\halo^{\circ (n+1)} H=\halo(\halo^{\circ n} H)$.

\begin{corollary}\label{cor:boundonProfIterated}
Let $H$ be a finitely generated amenable group. Let $\halo H$ be a naturally generated and large-scale commutative halo product having consistent blocks. Let $\varphi\colon x\mapsto x\cdot \Lambda_{\halo H}(x)$ and where $\Lambda_{\halo H}$ is the lamp growth sequence of $\halo H$. Then, given $p\geq 1$ and a integer $n\geq 0$, we have
\begin{equation*}
    \frac{1}{C}\profp{H}\left (\Psi_C^{\circ n}(x)\right )\leq\profp{\halo^{\circ n} H}(x)\leq C\prof{H}(C\ln^{\circ n}{x})
\end{equation*}
for some constant $C>0$, with $\Psi_C\colon x\mapsto \frac{1}{C}\varphi^{-1}(x)$.
\end{corollary}

\begin{proof}
We prove the statement by induction over $n\geq 0$. The case $n=0$ is trivial, and the case $n=1$ follows from Corollary~\ref{cor:boundonProf}. Now given $n\geq 2$, assume that the result holds true for $n-1$. There exists $C_1$ such that 
\begin{equation*}
    \frac{1}{C_1}\profp{H}\left (\Phi_{C_1}^{\circ (n-1)}(x)\right )\leq\profp{\halo^{\circ (n-1)} H}(x)\leq C_1\prof{H}(C_1\ln^{\circ (n-1)}{x}).
\end{equation*}
Moreover Corollary~\ref{cor:boundonProf} applied to $\halo^{\circ (n-1)} H$ implies that there exists $C_2>0$ such that
\begin{equation*}
    \frac{1}{C_2}\profp{\halo^{\circ (n-1)} H}\left (\frac{1}{C_2}\varphi^{-1}(x)\right )\leq\profp{\halo^{\circ n} H}(x)\leq C_2\prof{\halo^{\circ (n-1)} H}(C_2\ln{x}).
\end{equation*}
Combining these two estimates, we get the result.
\end{proof}

\subsection{Lampshufflers and lampjugglers}

We now apply our estimates on isoperimetric profiles to concrete examples, using~\cite[Facts~7.12-7.16]{GT24a} that computes the lamp growth sequences of most examples of halo products we are interested in. In this section, we address the case of lampshufflers, lampjugglers and their iterated versions. 

\smallskip

Let us recall that $\ell^p$-isoperimetric profiles of lampshufflers over polynomial growth groups are known:
\begin{equation*}
    \profp{\shuf{H}}(x) \simeq \left(\frac{\ln(x)}{\ln(\ln(x))}\right)^{\frac{1}{d}}
\end{equation*}
for any $p\ge 1$, when $H$ has growth degree $d\ge 1$. This can be directly deduced from~\cite{SCZ21} and~\cite{EZ21}.

\smallskip

From our work, we can deduce the following bounds on profiles of iterated lampjugglers. 

\begin{corollary}\label{cor:encadrementduprofildeshuf}
Let $p\ge 1$ and let $n\geq 0$ be an integer. Let $H$ be a finitely generated amenable group. Let $s\ge 1$ be an integer. Then the $\ell^p$-isoperimetric profile of $\jugglern{n}{s}{H}$ satisfies 
\begin{equation*}
    \profp{H}\left(\frac{1}{C}\frac{\ln^{\circ n}(x)}{\ln^{\circ (n+1)}(x)}\right) \preccurlyeq \profp{\jugglern{n}{s}{H}}(x) \preccurlyeq \prof{H}(C\ln^{\circ n}(x))
\end{equation*}
for some positive constant $C>0$.
\end{corollary}

\begin{proof}
From Corollary~\ref{cor:boundonProfIterated}, we know that 
\begin{equation}
    \profp{H}(\Psi^{\circ n}_{C_1}(x)) \preccurlyeq \profp{\jugglern{n}{s}{H}}(x)
\end{equation}
for some constant $C_1>0$, where $\Psi_{C_1}(x)=\frac{1}{C_1}\varphi^{-1}(x)$
and $\varphi(x)=x\cdot \Lambda_{\juggler{s}{H}}(x)=x\cdot (sx)!$. It remains to find the asymptotics of $\varphi^{-1}(x)$. We have by definition $\varphi^{-1}(x)(s\varphi^{-1}(x))!=x$, and from Stirling's formula we know that
\begin{equation*}
    \varphi^{-1}(x)(s\varphi^{-1}(x))! \sim \varphi^{-1}(x)\left(\frac{s\varphi^{-1}(x)}{e}\right)^{s\varphi^{-1}(x)}\sqrt{2\pi\cdot s\varphi^{-1}(x)}.
\end{equation*}
Taking the logarithm yields $\ln(x)=\ln\left(\varphi^{-1}(x)(s\varphi^{-1}(x))!\right) \sim s\varphi^{-1}(x)\ln(s\varphi^{-1}(x))$ and taking the logarithm once more, it follows that $\ln(\ln(x)) \sim \ln(s\varphi^{-1}(x))$.
Combining these two equivalences, this gives 
\begin{equation*}
    s\varphi^{-1}(x)\sim\frac{\ln(x)}{\ln(\ln(x))}. 
\end{equation*}
We thus get $\Psi^{\circ n}_{C_1}(x)\sim\frac{1}{sC_1}\frac{\ln^{\circ n}(x)}{\ln^{\circ (n+1)}(x)}$. It remains to consider a positive constant $C>sC_1$ and the proof is complete.
\end{proof}

We finally move on to precise estimates, under mild assumptions.

\begin{theorem}\label{th:profilesoflampjugglers}
Let $p\geq 1$ and let $n\geq 0$ be an integer. Let $H$ be a finitely generated amenable group such that $\profp{H}(x)\simeq \prof{H}(x)$. Then the following holds for the $\ell^p$-isoperimetric profile of $\jugglern{n}{s}{H}$.
\begin{itemize}
    \item If $H$ has polynomial growth of degree $d\geq 1$, then
    \begin{equation*}
        \profp{\jugglern{n}{s}{H}}(x)\simeq\left (\frac{\ln^{\circ n}(x)}{\ln^{\circ (n+1)}(x)}\right )^{\frac{1}{d}}.
    \end{equation*}
    \item If the $\ell^p$-isoperimetric profile $\profp{H}$ of $H$ satisfies Assumption~$(\star)$ and $\profp{H}\left(\frac{x}{\ln(x)}\right)\simeq \profp{H}(x)$, then
    \begin{equation*}
        \profp{\jugglern{n}{s}{H}}(x)\simeq\profp{H}(\ln^{\circ n}(x)).
    \end{equation*}
\end{itemize}
\end{theorem}

\begin{proof}
First of all, Assumption~$(\star)$ on the isoperimetric profiles enable us to erase the constants $\frac{1}{C}$ and $C$ in Corollary~\ref{cor:encadrementduprofildeshuf}.\par
Assume that $H$ has polynomial growth and $n=1$. The result for $s=1$ is already known (see~\cite{SCZ21} and~\cite{EZ21}). For $s>1$, we use Corollary~\ref{cor:boundonProfIterated} to get $\left(\frac{\ln(x)}{\ln(\ln(x))}\right)^{\frac{1}{d}}$ as a lower bound, and the fact that $\shuf{H}$ is a subgroup of $\juggler{s}{H}$ implies that we get the same estimate as an upper bound. Now that we have completed the case $n=1$, the general case $n>1$ is an immediate consequence of Corollary~\ref{cor:boundonProfIterated} applied to $n-1$ instead of $n$, $\juggler{s}{H}$ instead of $H$.\par
The second point of the statement is a direct consequence of Corollary~\ref{cor:boundonProfIterated}.
\end{proof}

Thus, lampjuggler groups often have the same $\ell^p$-isoperimetric profile as lampshufflers, even if the two are not quasi-isometric. For instance, if $H=\Z^{d}\wr\text{BS}(1,n)$, $d\ge 1$, $n\ge 2$, then $\shuf{H}$ and $\juggler{s}{H}$ are not quasi-isometric (by~\cite[Theorem~8.7]{GT24a} and~\cite[Corollary~1.4]{Dum25}) but both have $\ell^p$-isoperimetric profile $\simeq \profp{H}(\ln(x)) \simeq \ln(\ln(\ln(x)))$.

\smallskip

Let us record several examples of application of Theorem~\ref{th:profilesoflampjugglers}.
\begin{itemize}
    \item Solvable Baumslag-Solitar groups $\text{BS}(1,n)$, $n\ge 2$, have $\ell^p$-isoperimetric profile $\simeq \ln(x)$. Thus $\shuf{\text{BS}(1,n)}$ has $\ell^p$-isoperimetric profile $\simeq \ln(\ln(x))$. The same applies for lamplighters $F\wr\Z$, where $F$ is a non-trivial finite group;
    \item more generally, lamplighters $F\wr \Z^d$, where $d\ge 1$ and $F$ is non-trivial and finite, have $\ell^p$-isoperimetric profile $\simeq \ln(x)^{\frac{1}{d}}$. In this case, we get that 
    \begin{equation*}
        \profp{\shuf{F\wr\Z^d}}(x) \simeq \ln(\ln(x))^{\frac{1}{d}}. 
    \end{equation*}
    \item for $d\ge 1$, the group $H=\Z\wr\Z^d$ has $\ell^p$-isoperimetric profile $\simeq \left(\frac{\ln(x)}{\ln(\ln(x))}\right)^{\frac{1}{d}}$, so that 
    \begin{equation*}
        \profp{\shuf{\Z\wr\Z^d}}(x) \simeq \left(\frac{\ln(\ln(x))}{\ln(\ln(\ln(x)))}\right)^{\frac{1}{d}}.
    \end{equation*}
\end{itemize}

\begin{example}\label{ex:IsoProfBrieusselZheng}
For any non-decreasing function $f\colon \R_{+}\longrightarrow \R_{+}$ such that $x\longmapsto \frac{x}{f(x)}$ is non-decreasing, Brieussel and Zheng constructed in~\cite[Theorem~1.1]{BZ21} a finitely generated group $H$ with exponential volume growth having $\ell^p$-isoperimetric profile $\profp{H}(x)\simeq \frac{\ln(x)}{f(\ln(x))}$. It is proved in~\cite{Corr24} that $\profp{H}$ satisfies Assumption~$(\star)$ (see the discussion right after Corollary 4.1 in~\cite{Corr24}). Lastly, $H$ also satisfies 
\begin{equation*}
    \profp{H}\left(\frac{x}{\ln(x)}\right) \simeq \prof{H}(x).
\end{equation*}
This follows from the fact that asymptotic equivalence is preserved by $f$: if $g,h\colon \R_{+}\longrightarrow\R_{+}$ are equivalent, then $(1-\varepsilon)h(x)\le g(x)\le (1+\varepsilon)h(x)$ for some $\varepsilon>0$ and for large enough $x$, so that
\begin{equation*}
    f((1-\varepsilon)h(x)) \le f(g(x)) \le f((1+\varepsilon)h(x))
\end{equation*}
since $f$ is non-decreasing. Since $(1-\varepsilon)h(x) \le h(x) \le (1+\varepsilon)h(x)$, one has also
\begin{equation*}
    \frac{(1-\varepsilon)h(x)}{f((1-\varepsilon)h(x))} \le \frac{h(x)}{f(h(x))} \le \frac{(1+\varepsilon)h(x)}{f((1+\varepsilon)h(x))}.
\end{equation*}
since $x\longmapsto\frac{x}{f(x)}$ is non-decreasing, and thus 
\begin{equation*}
    (1-\varepsilon)f(h(x)) \le f((1-\varepsilon)h(x)) \le f(g(x)) \le f((1+\varepsilon)h(x)) \le (1+\varepsilon)f(h(x))
\end{equation*}
for all large enough $x$, whence $f(g(x)) \sim f(h(x))$. Thus, we may apply Theorem~\ref{th:profilesoflampjugglers}, and we obtain that 
\begin{equation*}
    \profp{\shuf{H}}(x) \simeq \frac{\ln(\ln(x))}{f(\ln(\ln(x)))}. 
\end{equation*}
\end{example}

\begin{question}\label{q:profile}
Is it true that all finitely generated amenable groups which do not have polynomial growth satisfy $\profp{H}\left(\frac{x}{\ln(x)}\right) \simeq \profp{H}(x)$?

\end{question}

\begin{remark}\label{rem:IteratedJuggler}
The same strategy shows that if $H$ is a finitely generated amenable group whose isoperimetric profiles satisfy Assumption~$(\star)$ and 
\begin{equation*}
    \profp{H}\left(\frac{x}{\ln(x)}\right) \simeq \profp{H}(x)\; \text{and} \;\profp{H}(x) \simeq \prof{H}(x),
\end{equation*}
then one has 
\begin{equation*}
    \profp{    \juggler{s_{1}}{\juggler{s_{2}}{\dots\juggler{s_{n}}{H}}}       }(x) \simeq \profp{H}(\ln^{\circ n}(x))
\end{equation*}
for any integers $n\ge 1$ and $s_{1},\dots,s_{n}\ge 1$, and any real number $p\ge 1$. 
\end{remark}

\begin{remark}\label{rm:EZnotoptimal}
In the particular case where $s=n=1$ and where $H$ is amenable with $\prof{H}(x)\simeq\ln^{\circ n}(x)$, we get $\prof{\shuf{H}}(x)\simeq\ln^{\circ (n+1)}(x)$. This means that the estimate
\begin{equation*}
    \fol{\shuf{H}}(x) \succcurlyeq V_{H}(x)^{V_{H}(x)}\simeq (e^{x})^{e^{x}}
\end{equation*}
obtained in~\cite[Corollary~1.4]{EZ21} in the case of an exponential growth group $H$ is not sharp.
\end{remark}

\subsection{Lampdesigners}\label{sec:lampdesigner}

Lampdesigners are close to lampjuggler groups, and in fact if $F$ is finite, $\designer{H}$ is a subgroup of $\juggler{|F|}{H}$, via the map 
\begin{align*}
    \begin{array}{cll}
    \designer{H} &\longrightarrow &\juggler{|F|}{H} \\
    ((f,\sigma),h) &\longmapsto &(\sigma', h)
    \end{array}
\end{align*}
where, given a pair $(f,\sigma)\in F\wr_{H}\fsym{H}$, $\sigma'$ is the permutation of $H\times F$ given by $\sigma'(h,i)=(\sigma(h), f(h)i)$. Hence, from Theorem~\ref{th:profilemonotonuousSubgroup}, we get a lower bound on $\ell^p$-isoperimetric profiles of lampdesigners, namely
\begin{equation*}
    \profp{\juggler{|F|}{H}}(x)\preccurlyeq \profp{\designer{H}}(x).
\end{equation*}
Additionally, note that $\designer{H}$ contains $\shuf{H}$ as a subgroup (and also as a quotient), hence
\begin{equation*}
    \profp{\designer{H}}(x) \preccurlyeq \profp{\shuf{H}}(x).
\end{equation*}

We immediately deduce the following exact estimates of the isoperimetric profiles of lampdesigners.

\begin{theorem}\label{th:profileoflampdesigners}
Let $p\geq 1$ and let $n\geq 0$ be an integer. Let $H$ be a finitely generated amenable group such that $\profp{H}(x)\simeq \prof{H}(x)$. Then the following holds for the $\ell^p$-isoperimetric profile of $\designern{n}{H}$.
\begin{itemize}
    \item If $H$ has polynomial growth of degree $d\geq 1$, then
    \begin{equation*}
        \profp{\designern{n}{H}}(x)\simeq\left (\frac{\ln^{\circ n}(x)}{\ln^{\circ (n+1)}(x)}\right )^{\frac{1}{d}}.
    \end{equation*}
    \item If the $\ell^p$-isoperimetric profile $\profp{H}$ of $H$ satisfies Assumption~$(\star)$ and $\profp{H}\left(\frac{x}{\ln(x)}\right)\simeq \profp{H}(x)$, then
    \begin{equation*}
        \profp{\designern{n}{H}}(x)\simeq\profp{H}(\ln^{\circ n}(x)).
    \end{equation*}
\end{itemize}
\end{theorem}

\subsection{Lampcloners and lampupcloners}

Let us now turn to lampcloners and lampupcloners over finite fields.

\begin{corollary}\label{cor:encadrementduprofildecloner}
Let $p\ge 1$ and let $n\geq 0$ be an integer. Let $H$ be a finitely generated amenable group and let $\field$ be a finite field. Then the $\ell^p$-isoperimetric profile of $\clonern{n}{H}$ satisfies 
\begin{equation*}
    \profp{H}\left(\frac{1}{C}\sqrt{\ln^{\circ n}{(x)}}\right) \preccurlyeq \profp{\clonern{n}{H}}(x) \preccurlyeq \prof{H}(C\ln^{\circ n}(x))
\end{equation*}
for some positive constant $C>0$. The same estimates hold for $\profp{\upcloner{\Z^d}}$, $d\geq 1$, where $\Z^d$ is endowed with the lexicographic order.
\end{corollary}

\begin{proof}
From Corollary~\ref{cor:boundonProfIterated}, we know that 
\begin{equation}
    \profp{H}(\Psi^{\circ n}_{C_1}(x)) \preccurlyeq \profp{\clonern{n}{H}}(x)
\end{equation}
for some constant $C_1>0$, where $\Psi_{C_1}(x)=\frac{1}{C_1}\varphi^{-1}(x)$
and $\varphi(x)=x\cdot \Lambda_{\cloner{H}}(x)$. It remains to find the asymptotics of $\varphi^{-1}(x)$.  By definition, we have that $\varphi^{-1}(x)\cdot\Lambda_{\cloner{H}}(\varphi^{-1}(x))=x$, and thus 
\begin{equation*}
    \ln(\varphi^{-1}(x))+\ln(\Lambda_{\cloner{H}}(\varphi^{-1}(x)))=\ln(x).
\end{equation*}
From~\cite[Fact~7.16]{GT24a}, $\ln(\Lambda_{\cloner{H}}(y)) \simeq y^2$, so the above equation tells us that 
\begin{equation*}
    \varphi^{-1}(x)^2 \simeq \ln(x)
\end{equation*}
whence $\varphi^{-1}(x) \simeq \sqrt{\ln(x)}$. We thus get $\Psi^{\circ n}_{C_1}(x)\sim\frac{1}{C_1}\sqrt{\ln^{\circ n}(x)}$. It remains to consider a positive constant $C>C_1$ and the proof is complete.\par
Proposition~\ref{prop:UpclonerFG} implies that we can apply Corollary~\ref{cor:boundonProfIterated} to $\upcloner{\Z^d}$, where $\Z^d$ is endowed with the lexicographic order. We also have $\ln{(\Lambda_{\upcloner{\Z^d}}(y))}\simeq y^2$, so the estimates remain the same.
\end{proof}

We finally move on to precise estimates, under mild assumptions.

\begin{theorem}\label{th:profilesoflampcloners}
Let $p\geq 1$ and let $n\geq 0$ be an integer. Let $H$ be a finitely generated amenable group such that $\profp{H}(x)\simeq \prof{H}(x)$. Let $\field$ be a field. Then the following holds for the $\ell^p$-isoperimetric profile of $\clonern{n}{H}$.
\begin{itemize}
    \item If $H$ has polynomial growth of degree $d\geq 1$, then
    \begin{equation*}
        \left (\ln^{\circ n}(x)\right)^{\frac{1}{2d}}\preccurlyeq\profp{\clonern{n}{H}}(x)\preccurlyeq\left (\ln^{\circ n}(x)\right )^{\frac{1}{d}}.
    \end{equation*}
    The same estimates hold for $\profp{\upcloner{\Z^d}}$, $d\geq 1$, where $\Z^d$ is endowed with the lexicographic order.
    \item If the $\ell^p$-isoperimetric profile $\profp{H}$ of $H$ satisfies Assumption~$(\star)$ and $\profp{H}\left(\sqrt{x}\right)\simeq \profp{H}(x)$, then
    \begin{equation*}
        \profp{\clonern{n}{H}}(x)\simeq\profp{H}(\ln^{\circ n}(x)).
    \end{equation*}
\end{itemize}
\end{theorem}

\begin{proof}
This is an immediate consequence of Corollary~\ref{cor:encadrementduprofildecloner}.
\end{proof}

This corollary applies to many groups that have slow profiles, for instance:
\begin{itemize}
    \item Baumslag Solitar groups $\text{BS}(1,n)$, $n\ge 2$, whose $\ell^p$-isoperimetric profile is $\simeq\ln(x)$. Thus $\cloner{\text{BS}(1,n)}$ has $\ell^p$-isoperimetric profile $\simeq \ln(\ln(x))$ for any $n\ge 2$ and finite field $\field$;
    \item The lamplighter $F\wr\Sigma$, where $\Sigma$ has polynomial growth of degree $d\ge 1$, has $\ell^p$-isoperimetric profile $\simeq (\ln(x))^{\frac{1}{d}}$, whence 
    \begin{equation*}
        \profp{\cloner{F\wr\Sigma}}(x) \simeq \ln(\ln(x))^{\frac{1}{d}}.
    \end{equation*}
\end{itemize}

Inspired by the case of lampshufflers, we would expect that, when $H$ has polynomial growth of degree $d\ge 1$, the $\ell^p$-isoperimetric profile of $\clonern{n}{H}$ is the lower bound that we found in the above statement, namely $\left (\ln^{\circ n}(x)\right )^{\frac{1}{2d}}$. Recall that for lampshufflers, we applied the upper bound from~\cite[Corollary~1.4]{EZ21} which is still optimal in the polynomial growth case, but its proof seems difficult to generalise for lampcloners.

\begin{remark}
We can slightly improve the upper bound in Theorem~\ref{th:profilesoflampcloners}, since for any group $H$, $\shuf{H}$ is a subgroup of $\cloner{H}$, considering the linear automorphisms permuting the vectors of the canonical basis provided by $H$. Hence, if $H$ has polynomial growth of degree $d\ge 1$, we have
\begin{equation*}
    \profp{\clonern{n}{H}}(x)\preccurlyeq\left(\frac{\ln^{\circ n}(x)}{\ln^{\circ (n+1)}(x)}\right)^{\frac{1}{d}}.
\end{equation*}
\end{remark}

Finally, Theorem~\ref{th:profilesoflampcloners} motivates a similar question as Question~\ref{q:profile}.

\begin{question}
Is it true that all finitely generated amenable groups which do not have polynomial growth satisfy $\profp{H}\left(\sqrt{x}\right) \simeq \profp{H}(x)$?
\end{question}

\section{Applications to quasi-isometric classifications and regular maps}\label{sec:appQIandRegularMaps} 

This section is dedicated to our applications about the existence of regular maps between halo products and their iterated versions. It relies on computations realised in Section~\ref{sec:computationsIsoProf}. In fact, for simplicity and conciseness, we will be focusing mainly on lampshufflers, but analogous statements can be derived for lampjugglers, lampdesigners, lampcloners and lampupcloners.

\smallskip 

Let us first distinguish iterated lampshufflers over amenable groups. 

\begin{corollary}\label{cor:shufflersoverslowprofiles}
Let $n,m\ge 1$. Let $H$ be a finitely generated amenable group. Assume that one of the following holds:
\begin{enumerate}[label=(\roman*)]
    \item\label{item:QIpolynomial} $H$ has polynomial growth of degree $d\ge 1$;
    \item\label{item:QInonpolynomial} 
    $\prof{H}$ satisfies Assumption~$(\star)$, $\prof{H}\left(\frac{x}{\ln(x)}\right) \simeq \prof{H}(x)$ and the following property for any integers $k,\ell\ge 0$:
    \begin{equation*}
        \prof{H}(\ln^{\circ k}(x)) \simeq \prof{H}(\ln^{\circ \ell}(x)) \Longrightarrow k=\ell.
    \end{equation*}
\end{enumerate}
Then $\shufn{n}{H}$ and $\shufn{m}{H}$ are quasi-isometric if and only if $n=m$.
\end{corollary}

\begin{proof}
Suppose that $\shufn{n}{H}$ and $\shufn{m}{H}$ are quasi-isometric. In particular, their isoperimetric profiles are asymptotically equivalent, and if $H$ has polynomial growth, we get
\begin{equation*}
    \left(\frac{\ln^{\circ n}(x)}{\ln^{\circ (n+1)}(x)}\right)^{\frac{1}{d}} \simeq \left(\frac{\ln^{\circ m}(x)}{\ln^{\circ (m+1)}(x)}\right)^{\frac{1}{d}}
\end{equation*}
by Theorem~\ref{th:profilesoflampjugglers}, which in turn implies $n=m$. If we are in case~\textit{\ref{item:QInonpolynomial}}, then $\shufn{n}{H}$ has $\ell^1$-isoperimetric profile $\simeq \prof{H}(\ln^{\circ n}(x))$ and $\shufn{m}{H}$ has $\ell^1$-isoperimetric profile $\simeq \prof{H}(\ln^{\circ m}(x))$. Thus $n=m$ using our assumption, and we are done. 
\end{proof}

In practice, assumptions of~\textit{\ref{item:QInonpolynomial}} are easy to check. It holds for instance for any amenable group whose profile is of the form $\prof{H}(x)\simeq \left (\ln^{\circ k}(x)\right )^{\alpha}$ for $\alpha>0$ and $k\ge 0$, such as solvable Baumslag-Solitar groups or lamplighters over polynomial growth groups.

\smallskip

In fact, the isoperimetric profile being monotonuous under regular maps between finitely generated amenable groups, we get more generally:

\begin{corollary}
Let $n,m\ge 1$. Let $H$ be a finitely generated amenable group whose isoperimetric profile $\prof{H}$ satisfies Assumption~$(\star)$. Suppose that $\prof{H}\left(\frac{x}{\ln(x)}\right) \simeq \prof{H}(\ln(x))$ and the following property holds for any integers $k,\ell\ge 0$:
\begin{equation*}
    \prof{H}(\ln^{\circ \ell}(x)) \preccurlyeq \prof{H}(\ln^{\circ k}(x)) \Longrightarrow k\leq\ell.
\end{equation*}
Then there exists a regular map from $\shufn{n}{H}$ to $\shufn{m}{H}$ if and only if $n\le m$.\qed
\end{corollary}

We have similar consequences at the other side of the spectrum:

\begin{corollary}\label{cor:IteratedShufflersPolynomialQIBiLip}
Let $n,m\ge 0$. Let $A$ and $B$ be infinite virtually abelian finitely generated groups, with growth degrees $a$ and $b$ respectively. Then the following are equivalent:
\begin{enumerate}[label=(\roman*)]
    \item\label{item:1QIbilip} $\shufn{n}{A}$ and $\shufn{m}{B}$ are quasi-isometric.
    \item\label{item:2QIbilip} $n=m$ and $a=b$.
    \item\label{item:3QIbilip} $\shufn{n}{A}$ and $\shufn{m}{B}$ are biLipschitz equivalent.
\end{enumerate}
\end{corollary}

\begin{proof}
The implication~\textit{\ref{item:3QIbilip}} $\Longrightarrow$~\textit{\ref{item:1QIbilip}} is obvious. 

\smallskip

\noindent We prove~\textit{\ref{item:1QIbilip}} $\Longrightarrow$~\textit{\ref{item:2QIbilip}}. Assume that $\shufn{n}{A}$ and $\shufn{m}{B}$ are quasi-isometric, so that they have asymptotically equivalent isoperimetric profiles. By Theorem~\ref{th:profilesoflampjugglers}, we then have 
\begin{equation}\label{eq:ComparisonProfile}
    \left(\frac{\ln^{\circ n}(x)}{\ln^{\circ (n+1)}(x)}\right)^{\frac{1}{a}} \simeq \left(\frac{\ln^{\circ m}(x)}{\ln^{\circ (m+1)}(x)}\right)^{\frac{1}{b}}
\end{equation}
and taking the logarithm, it follows that 
\begin{equation*}
    \ln\left(\frac{\ln^{\circ n}(x)}{\ln^{\circ (n+1)}(x)}\right) \simeq \ln\left(\frac{\ln^{\circ m}(x)}{\ln^{\circ (m+1)}(x)}\right).
\end{equation*}
The left-hand side is equivalent to $\ln^{\circ (n+1)}(x)$ and the right-hand side is equivalent to $\ln^{\circ (m+1)}(x)$, so that $n+1=m+1$, i.e. $n=m$. Re-injecting this information in~\eqref{eq:ComparisonProfile} now implies that $a=b$. 

\smallskip

\noindent If $n=m$ and $a=b$, then $A$ and $B$ are both biLipschitz equivalent to $\Z^{a}$~\cite[Claim 5.4]{Dum25}, and thus $A$ and $B$ are biLipschitz equivalent. Thus, by~\cite[Lemma~8.8]{GT24a}, there is a biLipschitz equivalence from $\shuf{A}$ to $\shuf{B}$. Iterating this, we get a biLipschitz equivalence 
\begin{equation*}
    \shufn{n}{A}\longrightarrow \shufn{n}{B}
\end{equation*}
as claimed. This shows~\textit{\ref{item:2QIbilip}} $\Longrightarrow$~\textit{\ref{item:3QIbilip}} and concludes the proof.
\end{proof}

\begin{remark}
For the broader class of virtually nilpotent groups, some implications still hold and some may fail. For instance,~\textit{\ref{item:1QIbilip}} $\Longrightarrow$~\textit{\ref{item:2QIbilip}} remains true, but the converse is false. For instance, $\Z^4$ and the Heisenberg group $H$ over $\Z$ both have growth degree $4$, but $\shuf{\Z^4}$ and $\shuf{H}$ are not quasi-isometric by~\cite[Corollary~8.9]{GT24a}, since $\Z^4$ and $H$ are not biLipschitz equivalent (e.g. they have different asymptotic dimensions).  
\end{remark}

For more general maps (e.g. regular maps), the isoperimetric profile is not sufficient to detect a constraint on polynomial growth degrees. However, asymptotic dimension does provide an inequality since, if $A$ has finite asymptotic dimension, then $\text{asdim}(\shuf{A})=\text{asdim}(A)$. Indeed, since $A$ is a subgroup of $\shuf{A}$, one has $\text{asdim}(A) \le \text{asdim}(\shuf{A})$, and on the other hand, since $\shuf{A}$ fits into a short exact sequence with kernel $\fsym{A}$, whose asymptotic dimension is $0$ as it is locally finite, and quotient $A$, one also has $\text{asdim}(\shuf{A}) \le \text{asdim}(A)$ (we refer the reader to the nice survey~\cite{BD08} for all these facts on asymptotic dimension). Iterating, we get 
\begin{equation*}
    \text{asdim}(\shufn{n}{A})=\text{asdim}(A)
\end{equation*}
for all $n\ge 0$. 

Note also that, if $A$ is virtually abelian, then its asymptotic dimension coincides with its growth degree. 

\begin{corollary}\label{cor:IteratedShufflersPolynomialRegularMap}
Let $n$ and $m$ be two natural integers. Let $A$ and $B$ be infinite virtually abelian finitely generated groups, with growth degrees $a$ and $b$ respectively. If there exists a regular map 
\begin{equation*}
    \shufn{n}{A}\longrightarrow \shufn{m}{B}
\end{equation*}
then $n\le m$ and $a\le b$. 
\end{corollary}

\begin{proof}
Assume that such a map exists. If $n=0$ there is nothing to prove, so we assume that $n\ge 1$. In this case, we cannot have $m=0$, because a group of exponential growth cannot regularly embed into a polynomial growth group. Hence $m\ge 1$ as well. Now, by Theorem~\ref{th:profilemonotonuous} and Theorem~\ref{th:profilesoflampjugglers}, one has 
\begin{equation}\label{eq:ComparisonProfile2}
    \left(\frac{\ln^{\circ m}(x)}{\ln^{\circ (m+1)}(x)}\right)^{\frac{1}{b}} \preccurlyeq \left(\frac{\ln^{\circ n}(x)}{\ln^{\circ (n+1)}(x)}\right)^{\frac{1}{a}}
\end{equation}
and taking the logarithm implies 
\begin{equation*}
    \ln\left(\frac{\ln^{\circ m}(x)}{\ln^{\circ (m+1)}(x)}\right) \preccurlyeq \ln\left(\frac{\ln^{\circ n}(x)}{\ln^{\circ (n+1)}(x)}\right).
\end{equation*}
The left-hand side is $\simeq \ln^{\circ(m+1)}(x)$, and the right-hand side is $\simeq \ln^{\circ (n+1)}(x)$, so it follows that $n+1\le m+1$, i.e. $n\le m$. Additionally, since asymptotic dimension is monotonuous under regular maps one gets 
\begin{equation*}
    a=\text{asdim}(A)=\text{asdim}(\shufn{n}{A})\le \text{asdim}(\shufn{m}{B})=\text{asdim}(B)=b
\end{equation*}
as claimed. 
\end{proof}

\begin{remark}\label{rm:extensionsfornilpotentgroups}
On the other hand, for more general polynomial growth groups $A$ and $B$, we can only conclude that the 
existence of a quasi-isometry between $\shufn{n}{A}$ and $\shufn{m}{B}$ imposes $n=m$, $a=b$ and $\text{asdim}(A)=\text{asdim}(B)$, and the existence of a regular map 
\begin{equation*}
    \shufn{n}{A}\longrightarrow \shufn{m}{B}
\end{equation*}
implies $n\le m$ and $\text{asdim}(A) \le \text{asdim}(B)$. 
\end{remark}

As an immediate consequence of Corollary~\ref{cor:IteratedShufflersPolynomialRegularMap}, we have the following.

\begin{corollary}\label{cor:QIRegularMap}
Let $n,m\ge 0$. Let $A$ and $B$ be infinite virtually abelian finitely generated groups, with growth degrees $a$ and $b$ respectively. Then the following are equivalent:
\begin{enumerate}[label=(\roman*)]
    \item the three equivalent assertions of Corollary~\ref{cor:IteratedShufflersPolynomialQIBiLip} are satisfied;
    \item there exist a regular map from $\shufn{n}{A}$ to $\shufn{m}{B}$, and a regular map from $\shufn{m}{B}$ to $\shufn{n}{A}$.\qed
\end{enumerate}
\end{corollary}

\smallskip

Thus, asymptotic dimension is an obstruction to the existence of a regular map $\shuf{\Z^d}\longrightarrow \shuf{\Z^k}$ when $d>k$. Hence, in the spirit of~\cite[Question~5.4]{BST12}, a natural question arises: can we also rule out the existence of such maps if we increase the asymptotic dimension of the target space, for instance with a polynomial growth factor? It turns out that the answer is positive, and that the isoperimetric profile still gives an obstruction, whereas asymptotic dimension becomes inefficient.

\smallskip

Indeed, thanks to the following lemma, under the assumption that $\prof{H}\succcurlyeq \prof{G}$, we have general estimates on the isoperimetric profile of $G\times H$ in terms of $\prof{G}$ and $\prof{H}$.

\begin{lemma}\label{lem:profileofdirectproducts}
Let $G$ and $H$ be finitely generated amenable groups. If $\prof{H}(n)\succcurlyeq \prof{G}(n)$, then one has 
\begin{equation*}
    \prof{G}(\sqrt{n}) \preccurlyeq \prof{G\times H}(n) \preccurlyeq \prof{G}(n).
\end{equation*}
\end{lemma}

To prove it, we need the following general fact.

\begin{lemma}\label{lem:UsefulLemmaProfile}
Let $f,g,\varphi\colon \left[1,+\infty\right[\longrightarrow \left[1,+\infty\right[$ be three non-decreasing maps, with $\varphi$ injective and satisfying $\varphi(x)\underset{x\to +\infty}{\to}+\infty$. Assume that there exists a positive constant $D>0$ such that
\begin{equation*}
        g(x)\preccurlyeq f(D\varphi(x)).
\end{equation*}
Then we have
\begin{equation*}
        g(K\varphi^{-1}(x))\preccurlyeq f(x)
\end{equation*}
for some positive constant $K>0$. If furthermore $g$ satisfies Assumption~$(\star)$, then we have
\begin{equation*}
        g(\varphi^{-1}(x))\preccurlyeq f(x).
\end{equation*}
\end{lemma}

\begin{proof}
By assumption, there exists a positive constant $C>0$ such that
\begin{equation*}
    g(x)\le Cf(D\varphi(Cx))
\end{equation*}
for all $x$ large enough. Let $x$ be a real number greater than $D\varphi(C)$, and let $n\geq 1$ be an integer such that $D\varphi(Cn)\le x\le D\varphi(C(n+1))$. Then we have
\begin{equation*}
    Cf(x)\ge Cf(D\varphi(Cn))\ge g(n)\ge g\left(\frac{n+1}{2}\right)\ge g\left (\frac{\varphi^{-1}(\frac{x}{D})}{2C}\right),
\end{equation*}
which can be reformulated as
\begin{equation*}
    g\left(K\varphi^{-1}(y)\right)=O\left(f(Dy)\right),
\end{equation*}
taking $y=\frac{x}{D}$ and $K=\frac{1}{2C}$. This shows the first part of the statement. Additionally, if $g$ satisfies Assumption~$(\star)$, then we have
\begin{equation*}
g\left(\varphi^{-1}\left(y\right)\right)=g\left(\frac{1}{K}\cdot K\varphi^{-1}(y)\right)=O\left(g\left(K\varphi^{-1}(y)\right)\right)=O\left(f(Dy)\right),
\end{equation*}
which concludes the proof.
\end{proof}

\begin{proof}[Proof of Lemma~\ref{lem:profileofdirectproducts}]
The estimate $\prof{G\times H}(n) \preccurlyeq \prof{G}(n)$ is a consequence of the fact that $G$ is a subgroup of $G\times H$ and Theorem~\ref{th:profilemonotonuousSubgroup}. Let us focus on the other inequality. Fix $n\in\N$ and subsets $A_{n}\subset G$, $B_{n}\subset H$ that realise $\prof{G}(n)$ and $\prof{H}(n)$ respectively, i.e. $|A_{n}|, |B_{n}|\le n$ and 
\begin{equation*}
    \prof{G}(n)=\frac{\left|A_{n}\right|}{\left|\partial_{G} A_{n}\right|},\; \prof{H}(n)=\frac{\left|B_{n}\right|}{\left|\partial_{H} B_{n}\right|}.
\end{equation*}
Then $A_{n}\times B_{n}\subset G\times H$ has cardinality $\le n^2$, and its boundary is given by 
\begin{equation*}
    \partial_{G\times H}(A_{n}\times B_{n}) = (\partial_{G}A_{n}\times B_{n})\cup(A_{n}\times \partial_{H}B_{n})
\end{equation*}
whence $\left|\partial_{G\times H}(A_{n}\times B_{n})\right| \le |\partial_{G}A_{n}|\cdot|B_{n}|+|A_{n}|\cdot|\partial_{H}B_{n}|$. Thus one gets 
\begin{equation*}
    \frac{\left|\partial_{G\times H}(A_{n}\times B_{n})\right|}{|A_{n}\times B_{n}|} \le \frac{|\partial_{G}A_{n}|\cdot|B_{n}|+|A_{n}|\cdot|\partial_{H}B_{n}|}{|A_{n}||B_{n}|}=\frac{|\partial_{G}A_{n}|}{|A_{n}|}+\frac{|\partial_{H}B_{n}|}{|B_{n}|}
\end{equation*}
and it follows that 
\begin{equation*}
    \prof{G\times H}(n^2) \ge \frac{|A_{n}\times B_{n}|}{\left|\partial_{G\times H}(A_{n}\times B_{n})\right|}\ge \frac{1}{\frac{|\partial_{G}A_{n}|}{|A_{n}|}+\frac{|\partial_{H}B_{n}|}{|B_{n}|}}=\frac{1}{\frac{1}{\prof{G}(n)}+\frac{1}{\prof{H}(n)}}.
\end{equation*}
Now, using Lemma~\ref{lem:UsefulLemmaProfile}, there exists a positive constant $K>0$ such that
\begin{equation*}
    \prof{G\times H}(x)\succcurlyeq\frac{1}{\frac{1}{\prof{G}\left(K\sqrt{x}\right)}+\frac{1}{\prof{H}\left(K\sqrt{x}\right)}}=\frac{1}{\frac{1}{\prof{G}\left(\sqrt{K^2x}\right)}+\frac{1}{\prof{H}\left(\sqrt{K^2x}\right)}}\succcurlyeq \frac{1}{\frac{1}{\prof{G}(\sqrt{x})}+\frac{1}{\prof{H}(\sqrt{x})}}
\end{equation*}
for all $x$ large enough. By assumption, $\prof{H}\succcurlyeq \prof{G}$, so that 
\begin{equation*}
    \prof{G\times H}\succcurlyeq \frac{1}{\frac{2}{\prof{G}(\sqrt{\cdot})}} \simeq \prof{G}(\sqrt{\cdot})
\end{equation*}
as claimed.
\end{proof}

Thus, if additionally the isoperimetric profile of $G$ satisfies $\prof{G}(\sqrt{\cdot})\simeq \prof{G}(\cdot)$, then $\prof{G\times H}\simeq \prof{G}$. This happens for many groups $G$ that have slow enough profiles, for instance:
\begin{itemize}
    \item Any polycyclic group with exponential growth, and more generally any GES group with exponential growth~\cite[Corollary~5]{Tes13};
    \item $F\wr\Sigma$, or $\shuf{\Sigma}$, where $F$ is finite and $\Sigma$ has polynomial growth;
    \item $\shufn{n}{H}$, where $H$ has profile $\prof{H}(x)\simeq \left (\ln^{\circ k}(x)\right )^{\alpha}$, for some $\alpha>0$ and integer $k\ge 1$.
\end{itemize}

As a concrete example, we have for instance: 

\begin{corollary}
Let $d,k,p\ge 1$ be three integers. There exists a regular map 
\begin{equation*}
    \shuf{\Z^d} \longrightarrow \Z^{p}\times\shuf{\Z^k}
\end{equation*}
if and only if $d\le k$.\qed
\end{corollary}

Finally, we want to compare lampshufflers to lamplighters. We already know from Proposition~\ref{prop:A2} that a wreath product over a subgroup of $H$ coarsely embeds into $\shuf{H}$. In~\cite[Corollary~7.17]{GT24a}, given two groups $H$ and $G$ satisfying some mild assumptions, it is proved that $\shuf{H}$ does not quasi-isometrically or coarsely embed into a lamplighter $E\wr G$. Our computations allow us to prove an iterated version of this result for free abelian groups: there is no regular map 
\begin{equation*}
\shufn{n}{\Z^d}\longrightarrow \Z/2\Z\wr(\Z/2\Z\wr(\dots(\Z/2\Z\wr \Z^d)))
\end{equation*}
where the wreath product is iterated $n$ times. In fact, we have more generally the following.

\begin{corollary}\label{cor:iteratedshufintoiteratedlamplighter}
Let $G$ and $H$ be finitely generated amenable groups. Suppose that there is a regular map 
\begin{equation*}
\shufn{n}{H}\longrightarrow \Z/2\Z\wr(\Z/2\Z\wr(\dots\wr(\Z/2\Z\wr G))),
\end{equation*}
where the wreath product is iterated $n$ times, $n\ge 1$. Then the following holds.
\begin{enumerate}[label=(\roman*)]
      \item\label{item:1shufintolamp} If $H$ has polynomial growth of degree $d\ge 1$, then
      \begin{equation*}
         \prof{G}(x)\preccurlyeq\frac{\prof{H}(x)}{\left(\ln(x)\right)^{\frac{1}{d}}};
      \end{equation*}
      \item\label{item:2shufintolamp} If $\prof{H}$ satisfies Assumption ~$(\star)$ and $\prof{H}\left(\frac{x}{\ln(x)}\right)\simeq\prof{H}(x)$, then
      \begin{equation*}
        \prof{G}(x)\preccurlyeq\prof{H}(x).
      \end{equation*}
\end{enumerate}
\end{corollary}

\begin{proof}
In case~\textit{\ref{item:1shufintolamp}}, we get
\begin{equation*}
        \prof{G}(\ln^{\circ n}(x))\preccurlyeq\left(\frac{\ln^{\circ n}(x)}{\ln^{\circ (n+1)}(x)}\right)^{\frac{1}{d}}=\frac{\prof{H}(\ln^{\circ n}(x))}{\left(\ln^{\circ (n+1)}(x)\right)^{\frac{1}{d}}}.
\end{equation*}
Using Assumption~$(\star)$ for the logarithm and for $\prof{H}$, we have
\begin{equation*}
        \prof{G}(\ln^{\circ n}(x))=O\left(\frac{\prof{H}(\ln^{\circ n}(x))}{\left(\ln^{\circ (n+1)}(x)\right)^{\frac{1}{d}}}\right)
\end{equation*}
and the change of variable $x'=\ln^{\circ n}(x)$ gives the result. In case~\textit{\ref{item:2shufintolamp}}, we get rather 
\begin{equation*}
    \prof{G}(\ln^{\circ n}(x))=O\left(\prof{H}(\ln^{\circ n}(x))\right)
\end{equation*}
and we conclude similarly.
\end{proof}

We get the following consequence in the case of polynomial growth groups.

\begin{corollary}\label{cor:iteratedshufintoiteratedlamplighterPolynomial}
Let $d,k,n\ge 1$, and let $G$ and $H$ be finitely generated groups of polynomial growth with growth degrees $k$ and $d$ respectively. If there is a regular map 
\begin{equation*}
\shufn{n}{H}\longrightarrow \Z/2\Z\wr(\Z/2\Z\wr(\dots\wr(\Z/2\Z\wr G)))
\end{equation*}
then $d<k$, where the wreath product is iterated $n$ times. 
\end{corollary}

Note that this statement cannot be reached with methods from~\cite{GT24a}, even for quasi-isometric or coarse embeddings, since the thick bigon property used in~\cite{GT24a} is not stable under iterations of lampshufflers.

\begin{proof}
By Corollary~\ref{cor:iteratedshufintoiteratedlamplighter}, we have $x^{\frac{1}{k}}\preccurlyeq\left(\frac{x}{\ln(x)}\right)^{\frac{1}{d}}$, which immediately implies $d<k$.
\end{proof}

Note that, if $G$ is a proper subgroup of $H$, then an iteration of Proposition~\ref{prop:A2} ensures that $\shufn{n}{H}$ contains $\Z/2\Z\wr(\Z/2\Z\wr(\dots(\Z/2\Z\wr G)))$ (iterated $n$ times) as a subgroup, and thus we get a regular map 
\begin{equation*}
    \Z/2\Z\wr(\Z/2\Z\wr(\dots\wr(\Z/2\Z\wr G))) \longrightarrow \shufn{n}{H}. 
\end{equation*}
\appendix\section{Lamplighter subgroups in lampshu{f}flers}\label{appendixA}

In the article, the strategy for getting optimal upper bounds on the isoperimetric profile of halo products is to find subgraphs that are quasi-isometric to lamplighter graphs. In this appendix, our aim is to prove that, under additional mild algebraic assumptions on the base group, we can find lamplighter \textit{subgroups} inside lampshufflers.

\smallskip

For other halo products, such as lampjugglers, lampdesigners and lampcloners, we already observed in Section~\ref{sec:halo} that they contain wreath products over the same base group as subgroups. 

\smallskip

Let us recall the following definition: we say that a non-decreasing map $h\colon\R_{+}\longrightarrow \R_{+}$ satisfies Assumption~$(\star)$ if 
\begin{equation*}
    \forall C>0,\; h(Cx)=O(h(x)). 
\end{equation*} 

\smallskip

The motivation for this strategy comes from the following result, due to Silva. In the upcoming result (Proposition~\ref{prop:A2}), we will use the main idea of its proof. 

\begin{proposition}[{\cite[Proposition~2.3]{Sil24}}]\label{prop:A1}
Let $H$ be an infinite non co-Hopfian group. Then, for any finite group $F$, $\shuf{H}$ has a subgroup isomorphic to $F\wr H$. 
\end{proposition}

Recall that a group $H$ is \textit{co-Hopfian} if any injective morphism $H\longrightarrow H$ is also surjective. Equivalently, a group is co-Hopfian if it has no proper subgroup isomorphic to itself.

\smallskip

Many amenable groups are known to be non co-Hopfian, including:
\begin{itemize}
    \item $\Z^d$, $d\ge 1$, and more generally any finitely generated abelian group;
    \item Some torsion-free nilpotent groups, such as the Heisenberg group over the integers~\cite{Corn16};
    \item Solvable Baumslag-Solitar groups $\text{BS}(1,n)$, $n\ge 1$~\cite{NP11};
    \item Wreath products $N\wr G$ where at least one of the two groups is not co-Hopfian~\cite[Proposition~6.1]{BFF24};
    \item Houghton's groups $H_{n}$, $n\ge 2$~\cite[Theorem~7.1]{BCMR16};
    \item The Grigorchuk's group~\cite{Lys85}.
\end{itemize}
In the non-amenable side, they also include for instance non-abelian free groups~\cite[section III.22]{dlH00} or right-angled Artin groups~\cite{Cas16}.

\smallskip

Thus, using Theorem~\ref{th:profilemonotonuousSubgroup}, it directly follows that for $H$ an amenable and non co-Hopfian group, one has already
\begin{equation*}
    \profp{\shuf{H}}(x) \preccurlyeq \prof{H}(\ln(x))
\end{equation*}
when $\prof{H}$ satisfies Assumption~($\star$).

\smallskip

In fact, we can derive from Silva's proof the following more general result.

\begin{proposition}\label{prop:A2}
Let $H$ be a group. If $K$ is a proper subgroup of $H$, then $\shuf{H}$ contains a subgroup isomorphic to $\fsym{[H:K]}\wr K$. 
\end{proposition}

\begin{proof}
Denote $[H:K]\defeq m \in \lbrace 2,3,\dots\rbrace\cup\lbrace \infty\rbrace$, and let $S\subset H$ be a set of representatives of $H/K$, so that $|S|=m$ and we have a partition 
\begin{equation*}
    H=\bigsqcup_{k\in K}kS.
\end{equation*}
Consider then 
\begin{equation*}
    G \defeq \big\lbrace (\sigma, h) \in\shuf{H} : h\in K, \;\sigma(k'S)=k'S\; \text{for all $k'\in K$}\big\rbrace.
\end{equation*}
It is not hard to check that $G$ is a subgroup of $\shuf{H}$, and given $\sigma\in\fsym{H}$ satisfying $\sigma(k'S)=k'S$ for all $k'\in K$, we can define a map 
\begin{align*}
    f_{\sigma}\colon
    \begin{array}{cll}    
    K &\longrightarrow &\fsym{m} \\
    k&\longmapsto &(k^{-1}\cdot\sigma)\big|_{S}
    \end{array}.
\end{align*}
Then, a direct computation shows that $f_{k\cdot \sigma}=k\cdot f_{\sigma}$ and $f_{\sigma\circ\tau}=f_{\sigma}f_{\tau}$ for any $k\in K$ and any $\sigma,\tau\in\fsym{H}$ satisfying $\sigma(k'S)=k'S$ for all $k'\in K$. This implies that the map 
\begin{align*}
    \begin{array}{ccl} G &\longrightarrow &\fsym{m}\wr K \\
    (\sigma, h)&\longmapsto &(f_{\sigma},h)
    \end{array}
\end{align*}
is a group isomorphism. This concludes the proof. 
\end{proof}

Therefore, one recovers the upper bound in Theorem~\ref{th:boundsForProfile intro} for lampshufflers.

\begin{corollary}\label{cor:A3}
Let $p\ge 1$. Let $H$ be an amenable finitely generated group with a finitely generated proper subgroup $K$ such that $\profp{H}(x)\simeq \profp{K}(x)$. If $\prof{H}$ satisfies Assumption~$(\star)$, then one has 
\begin{equation*}
    \profp{\shuf{H}}(x) \preccurlyeq \profp{H}(\ln(x)).
\end{equation*}
\end{corollary}

The corollary applies of course to any non co-Hopfian group, but beyond, also to any finitely generated group $H$ having at least one finite-index subgroup $K$. One algebraic criterion to ensure the latter is to be \textit{non-perfect}.

\begin{definition}
Let $H$ be a group. We say that $H$ is \textit{perfect} if $[H,H]=H$, where $[H,H]$ is the commutator subgroup of $H$.
\end{definition}

Examples of perfect groups include finite alternating groups $A_{n}$ for $n\ge 5$ and linear groups $\text{SL}(n,K)$ for $n\ge 3$ and non-commutative fields $K$. Among the three Thompson's groups $F\subset T\subset V$, $T$ and $V$ are simple, thus perfect, as well as the commutator subgroup $[F,F]$ of $F$~\cite{CFP96}. 

\smallskip

As mentioned above, we are in fact more interested in groups that are not perfect, due to the next statement.

\begin{lemma}\label{lm:finiteindexsubgroupsinnonperfectgroups}
Let $H$ be a finitely generated group which is not perfect. Then $H$ has a proper finite-index subgroup.
\end{lemma}

\begin{proof}
As $H$ is not perfect, $[H,H]$ is a proper subgroup and the quotient $H/[H,H]$ is a non-trivial abelian finitely generated group. It has therefore a proper finite-index subgroup, and lifting the latter provides a proper finite-index subgroup for $H$, that contains $[H,H]$.  
\end{proof}

On the amenable side, the class of non-perfect groups is huge. It includes for instance
\begin{itemize}
    \item All solvable groups, in particular nilpotent and polycyclic groups;
    \item Lamplighters, lampshufflers and lampjugglers over amenable non-perfect groups;
    \item More generally, any semi-direct product $N\rtimes Q$ where $N$ and $Q$ are amenable and $Q$ is non-perfect, as well as any subgroup of $N\rtimes Q$;
    \item Houghton's groups $H_{n}$, $n\ge 2$.
\end{itemize}

\smallskip 

Henceforth, for an amenable non-perfect group $H$, Corollary~\ref{cor:A3} directly provides an upper bound on the $\ell^p$-isoperimetric profiles of $\shuf{H}$.

\smallskip

We emphasize here that Corollary~\ref{cor:A3} can also cover cases that are not covered by Proposition~\ref{prop:A1}, since they are indeed examples of finitely generated torsion-free nilpotent groups that are co-Hopfian~\cite{Bel03}. 

\smallskip

It is also worth mentioning that there do exist amenable perfect groups. Such groups have been constructed by Juschenko and Monod in~\cite{JM13}, and are even simple. Thus Corollary~\ref{cor:A3} do not apply to them, and on the other hand it seems to be an open problem whether they are co-Hopfian or not; see the list of open questions in~\cite{Corn14}. Note also that we do not know the asymptotic behaviour of the isoperimetric profiles of these groups. 

\begin{remark}\label{rem:Preservation}
As pointed out above, lampshuffler groups over non-perfect groups are themselves non-perfect. This simply follows from the set inclusions 
\begin{equation*}
    [\shuf{H},\shuf{H}]\subset \fsym{H}\times [H,H] \subsetneq \shuf{H}.
\end{equation*}
In fact, it is true more generally that a lampshuffler $\shuf{H}$ over a non-trivial group $H$ is never perfect, regardless of the perfectness of $H$. Indeed, since the action of $H$ on $\fsym{H}$ preserves the parity of the permutations, the commutator subgroup of $\shuf{H}$ is contained in $\mathcal{A}(H)\times [H,H]$, where $\mathcal{A}(H)\subset\fsym{H}$ is the set of even permutations. Since $\mathcal{A}(H)$ is not equal to $\fsym{H}$ if $|H|\ge 3$, and since $[H,H]$ is trivial if $|H|=2$, we thus see that $\shuf{H}$ is not equal to its commutator subgroup. 
\end{remark}

Lastly, regarding lampshufflers, non co-Hopficity can also be used to deduce the isoperimetric profile of iterated lampshufflers, since it is stable under iteration of lampshufflers:

\begin{proposition}\label{prop:A7}
If $H$ is not co-Hopfian, then $\shuf{H}$ is not co-Hopfian. 
\end{proposition}

\begin{proof}
If $H$ is not co-Hopfian, fix an injective morphism $\psi\colon H\longrightarrow H$ which is not surjective, and define the map 
\begin{align*}
\varphi\colon
\begin{array}{cll}
\shuf{H}&\longrightarrow &\shuf{H} \\
(\sigma, g)&\longmapsto &(\overline{\sigma}, \psi(g))
\end{array}
\end{align*}
where, for any $\sigma\in \fsym{H}$, $\overline{\sigma}\in\fsym{H}$ is defined as 
\begin{align*}
    \overline{\sigma}\colon
    \begin{array}{cll}
    H&\longrightarrow &H \\
    g&\longmapsto &\begin{cases}\psi(\sigma(\psi^{-1}(g))) &\mbox{if $g$ is in the image of $\psi$} \\ g &\mbox{otherwise}\end{cases}
    \end{array}.
\end{align*}
Then one checks directly that the correspondence $\sigma\longmapsto\overline{\sigma}$ is well-defined and that the following two properties hold:
\begin{enumerate}[label=(\roman*)]
    \item $\overline{\sigma\circ \tau}=\overline{\sigma}\circ\overline{\tau}$ for any $\sigma,\tau\in\fsym{H}$;
    \item $\overline{p\cdot\sigma}=\psi(p)\cdot\overline{\sigma}$ for any $p\in H$ and $\sigma\in \fsym{H}$.
\end{enumerate}
These two points imply that $\varphi$ is a morphism, which is injective since $\psi$ is, and which is not surjective since $\psi$ is not.
\end{proof}

Going further, it is also natural to ask whether these algebraic assumptions are preserved under iteration of other halo products. First, it is immediate that a halo product over a non-perfect group is not perfect, since the definition of the product law of $\halo H$ directly gives
\begin{equation*}
    [\halo H,\halo H]\subset L(H)\times [H,H] \subsetneq \halo H,
\end{equation*}
as we pointed out in Remark~\ref{rem:Preservation}, in the special case of lampshufflers. In this remark, we also pointed out that lampshufflers were in fact never perfect, using the parity of the permutations, which is invariant by the action of the base group. Thus, for a general halo group $\halo H$, it is enough to find a non-trivial morphism from $L(H)$ to an abelian group, which is invariant by the action of the base group. For instance, for lampcloners, we can use the determinant of linear maps.

\smallskip

Regarding preservation of non-co-hopficity under iterations of halo products, we can only find proofs in concrete examples. For lampcloners, we can adapt the above proof for lampshufflers. Indeed, let us first note that a group morphism $\psi\colon H\longrightarrow H$ gives rise to a linear automorphism $\tilde{\psi}\colon V_{H}\longrightarrow V_{H}$, by permuting the vector of the canonical basis given by $H$, and if $\psi$ is injective but not bijective, then so is $\tilde{\psi}$. Finally, it remains to
define the map 
\begin{align*}
\varphi\colon
\begin{array}{cll}
\cloner{H}&\longrightarrow &\cloner{H} \\
(\sigma, g)&\longmapsto &(\overline{\sigma}, \psi(g))
\end{array}
\end{align*}
where, for any $\sigma\in \text{FGL}(H)$, $\overline{\sigma}\in \text{FGL}(H)$ is defined as 
\begin{align*}
    \overline{\sigma}\colon
    \begin{array}{cll}
    V_H&\longrightarrow &V_H \\
    v&\longmapsto &\begin{cases}\tilde{\psi}(\sigma(\tilde{\psi}^{-1}(v))) &\mbox{if $v\in V_{H}$ is in the image of $\tilde{\psi}$} \\ v &\mbox{otherwise}\end{cases}
    \end{array},
\end{align*}
and to check that $\varphi$ an injective, but not bijective, group morphism.

\printbibliography

{\bigskip
		\footnotesize
		
		\noindent C.~Correia, \textsc{Université Paris Cité, Institut de Mathématiques de Jussieu-Paris Rive Gauche, 75013 Paris, France}\par\nopagebreak\noindent
		\textit{E-mail address: }\texttt{corentin.correia@imj-prg.fr}}

{\bigskip
		\footnotesize
		
		\noindent V.~Dumoncel, \textsc{Université Paris Cité, Institut de Mathématiques de Jussieu-Paris Rive Gauche, 75013 Paris, France}\par\nopagebreak\noindent
		\textit{E-mail address: }\texttt{vincent.dumoncel@imj-prg.fr}}

\end{titlepage}
\end{document}